\documentclass{article}

\usepackage{amsmath}
\usepackage{amssymb}
\usepackage{amsthm}
\usepackage{tikz-cd}
\usepackage{mathrsfs}
\usepackage{bm}
\usepackage{comment}
\usepackage{tikz}

\usepackage{a4wide}

\newtheorem{thm}{Theorem}[section]
\newtheorem{lem}[thm]{Lemma}
\newtheorem{cor}[thm]{Corollary}
\newtheorem{dfn}{Definition}[section]

\begin{document}
\title{The computational complexity of the solid torus core recognition problem}
\author{Yuya Nishimura}
\date{}
\maketitle

\abstract{
    The solid torus core recognition problem is the problem that, given a knot in the solid tours, decides whether the knot is the core of the solid torus.
    That problem is in \textbf{NP} since the thickened torus recognition problem is in \textbf{NP}.
    We give an alternate proof of that fact and prove that the problem is in \textbf{co-NP}.
    It is also proved that the Hopf link recognition problem is in \textbf{NP} and \textbf{co-NP} as a corollary of this result.
    }

\section{Introduction}
The \textit{unknot recognition problem} is the problem of deciding whether the knot $K$ represented by a given knot diagram is the unknot in the $3$-sphere, 
namely, the problem of deciding whether $K$ has the diagram with no crossings.
This problem is one of the fundamental problems in the computational topology.
Haken showed in \cite{Haken} that there is an algorithm to solve the unknot recognition problem using normal surface theory.
With regard to the computational complexity of this problem,
Hass, Lagarias and Pippenger showed in \cite{HLP} that this problem is in \textbf{NP}, i.e. there is an non-deterministic polynomial time algorithm to solve the problem.
Moreover, it is proved by Lackenby in \cite{Lack} that the unknot recognition problem is in \textbf{co-NP}.
Thus, the unknot recognition problem is in \textbf{NP} $\cap$ \textbf{co-NP}.
However, it remains to be an open problem whether this problem is in \textbf{P}.

In this paper, we consider knots in the solid torus $V$.
The knots in $V$ that are non-affine and prime up to $6$ crossings are completely classified in \cite{BM}.
Here, a knot $K$ in $V$ is said to be \textit{non-affine} if there are no embedded $3$-ball in $V$ containing $K$.
The \textit{solid torus core recognition problem} is the problem that, given a knot $K$ in the solid torus $V$, decides whether $K$ is the core of $V$, 
namely, the problem of deciding whether $K$ is non-affine and has a diagram with no crossings.
A knot $K$ in the solid torus $V$ is the core of $V$ if and only if the exterior $V - \text{int}N(K)$ is homeomorphic to $T^2 \times [0,1]$, where $T^2$ is the torus.
Recently, Haraway and Hoffman announced in \cite{HH} that for every compact surface $\Sigma$, the $\Sigma \times [0,1]$ recognition problem is in \textbf{NP}.
This implies that the solid torus core recognition problem is in \textbf{NP}.
In this paper, we give an alternate proof of that fact. 
\begin{thm}\label{thm:main}
    The solid torus core recognition problem is in \textbf{NP}.
\end{thm}

Haraway and Hoffman also announced in \cite{HH} that for every compact surface $\Sigma$, the $\Sigma \times [0,1]$ recognition problem is in \textbf{co-NP} among orientable irreducible $3$-manifolds.
Using this theorem, we can show that the solid torus core recognition problem is in \textbf{co-NP}.
\begin{thm}\label{thm:main2}
    The solid torus core recognition problem is in \textbf{co-NP}.
\end{thm}

The \textit{Hopf link recognition problem} is the problem of deciding whether the link represented by a given link diagram is the Hopf link.
Assuming the generalized Riemann hypothesis, the Hopf link recognition problem is in \textbf{co-NP} (\cite{XZ}).
The link $L = K_1 \cup K_2$ in the $3$-sphere $\mathbb{S}^3$ is the Hopf link if and only if 
$K_1$ is the unknot, i.e. the exterior $\mathbb{S}^3 - \text{int}N(K_1)$ is the solid torus, and $K_2$ is the core of $\mathbb{S}^3 - \text{int}N(K_1)$.
Since the unknot recognition problem is in \textbf{NP} $\cap$ \textbf{co-NP}, 
we can prove that the Hopf link recognition problem is in \textbf{NP} $\cap$ \textbf{co-NP} as a corollary of Theorem \ref{thm:main} without assuming the generalized Riemann hypothesis.
\begin{cor}
    The Hopf link recognition problem is in \textbf{NP} $\cap$ \textbf{co-NP}.
\end{cor}
\textbf{Acknowledgments.}
The author would like to thank his supervisor Yuya Koda for helpful advice and encouragement.
He is also grateful to Neil Hoffman for pointing out that the solid torus core recognition and the Hopf link recognition problem are in \textbf{co-NP}.
This work was supported by JST, the establishment of university fellowships towards the creation of science technology innovation, Grant Number JPMJFS2129.

\section{Preliminaries}

\subsection{Knots in the solid torus}
A \textit{knot} $K$ in a $3$-manifold $M$ is a piecewise-linear simple closed curve embedded in $M$.
Two knots $K$ and $K'$ in $M$ are \textit{ambient isotopic} 
if there is a continuous map $F:M \times [0,1] \to M$ such that, if $f_t$ denotes $F|_{M \times \{t\}}$, $f_t:M \to M$ is a homeomorphism for each $t \in [0,1]$, $f_0$ is the identity map, and $f_1(K) = K'$.
Given a knot $K$ in a $3$-manifold $M$, the \textit{exterior} of $K$ is obtained from $M$ by removing the interior of the regular neighborhood $N(K)$ of $K$.

We consider knots in the solid torus $V = D^2 \times \mathbb{S}^1$, where $D^2$ is the disk and $\mathbb{S}^1$ is the circle.
Let $x$ be a point of the interior of $D^2$.
A knot $K$ is called the \textit{core} of $V$ if $K$ is ambient isotopic to the knot $\{x\} \times \mathbb{S}^1$ in $D^2 \times \mathbb{S}^1$.

\begin{dfn}
    Let $K$ be an oriented knot in the solid torus $V$.
    Fix an isomorphism $f$ from $H_1(V; \mathbb{Z})$ to $\mathbb{Z}$.
    Then the \textit{rotation number}, denoted by $r_f(K)$, of $K$ is defined by $f([K]) \in \mathbb{Z}$.
\end{dfn}
The absolute value of the rotation number does not depend on an orientation of a knot and an isomorphism from $H_1(V; \mathbb{Z})$ to $\mathbb{Z}$.
Thus, for a non-oriented knot $K$ in $V$, we denote by $|r(K)|$ the absolute value of the rotation number, where an orientation of $K$ and an isomorphism from $H_1(V; \mathbb{Z})$ to $\mathbb{Z}$ are auxiliarily fixed.
This is an invariant of non-oriented knots in $V$.
If a knot $K$ is the core of $V$, then we see that $|r(K)| = 1$.
However, $|r(K)| = 1$ does not necessarily mean that $K$ is the core of $V$ (See Figure \ref{fig:ex1}).
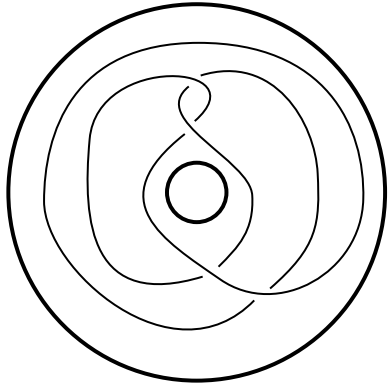
\begin{figure}[htpb]
    \centering
    \tikzset{every picture/.style={line width=0.75pt}} 

\begin{tikzpicture}[x=0.75pt,y=0.75pt,yscale=-1,xscale=1]

\draw    (97,49.5) .. controls (76,66.5) and (128,86.5) .. (129,104.5) .. controls (130,122.5) and (122,130.5) .. (112,140.5) ;
\draw    (100,67) .. controls (133,38) and (51,32.5) .. (47,76) .. controls (43,119.5) and (48,162.5) .. (104,145.5) ;
\draw    (95,73.5) .. controls (52,107.5) and (84,127) .. (113,147) .. controls (142,167) and (184,143.5) .. (185,106.5) .. controls (186,69.5) and (166,27.5) .. (102,27.5) .. controls (38,27.5) and (24,73.5) .. (24,107.5) .. controls (24,141.5) and (89,199.5) .. (130,157.5) ;
\draw    (103,44) .. controls (140,32.5) and (160.75,67.76) .. (162,93.5) .. controls (163.25,119.24) and (162,130.5) .. (138,151.5) ;
\draw  [line width=1.5]  (6,103) .. controls (6,50.53) and (48.53,8) .. (101,8) .. controls (153.47,8) and (196,50.53) .. (196,103) .. controls (196,155.47) and (153.47,198) .. (101,198) .. controls (48.53,198) and (6,155.47) .. (6,103) -- cycle ;
\draw  [line width=1.5]  (86,103) .. controls (86,94.72) and (92.72,88) .. (101,88) .. controls (109.28,88) and (116,94.72) .. (116,103) .. controls (116,111.28) and (109.28,118) .. (101,118) .. controls (92.72,118) and (86,111.28) .. (86,103) -- cycle ;

\end{tikzpicture}
    \caption{A knot $K$ with $|r(K)| = 1$ but not the core of the solid torus}
    \label{fig:ex1}
\end{figure}

\subsection{Triangulations}
Let $\Delta = \{\Delta_1, \ldots, \Delta_n\}$ be a collection of disjoint $n$ tetrahedra in $\mathbb{R}^3$.
A \textit{face-pairing} on $\Delta$ is an affine map between two distinct faces of tetrahedra $\Delta_i$ and $\Delta_j$ (possibly $i = j$).
Let $\mathcal{F}$ be a collection of face-pairings on $\Delta$ such that each face of the tetrahedra appears at most once.
Then the pair $(\Delta, \mathcal{F})$ is called a \textit{generalized triangulation}.
In this paper, we call a generalized triangulation simply a triangulation.
The \textit{underlying space}, denoted by $|\mathcal{T}|$, of a triangulation $\mathcal{T} = (\Delta, \mathcal{F})$ is the quotient space obtained by gluing the union of the tetrahedra by the face-pairings.
If $|\mathcal{T}|$ is homeomorphic to a $3$-manifold $M$, then $\mathcal{T}$ is called a \textit{triangulation of $M$}.
We abuse notation by writing $\Delta_i$ for the image of a tetrahedron $\Delta_i$ in $|\mathcal{T}|$.
The \textit{size}, denoted by $\text{size}(\mathcal{T})$, of a triangulation $\mathcal{T}$ is the number of tetrahedra of $\mathcal{T}$.
If $\text{size}(\mathcal{T}) = n$, then $\mathcal{T}$ is called an \textit{$n$-tetrahedra triangulation}.

An input of the solid torus core recognition problem is given as a pair of a triangulation of the solid torus and a knot in its $1$-skeleton.
\begin{dfn}[The solid torus core recognition problem]
    Let $\mathcal{T}_V$ be a triangulation of the solid torus $V$.
    Assume that a knot $K$ in $V$ is represented by a collection of edges of the $1$-skeleton of $\mathcal{T}_V$. 
    The solid torus core recognition problem is the problem that, given the pair $(\mathcal{T}_V, K)$, decides whether $K$ is the core of $V$.
\end{dfn}
Let $\mathcal{T}_V$ be an $n$-tetrahedra triangulation of the solid torus $V$.
By labeling the vertices of the tetrahedra of $\mathcal{T}_V$ by $1, \ldots, 4n$, each face-pairing of $\mathcal{T}_V$ is represented by a pair of triples of integers $((i_1, i_2, i_3), (j_1, j_2, j_3))$. 
Since the number of face-pairings of $\mathcal{T}_V$ is at most $4n$, the face-pairings of $\mathcal{T}_V$ is represented by at most $4n$ pairs of triples of integers.
In addition, by labeling the edges of $\mathcal{T}_V^{(1)}$ by integers, a knot $K$ in $\mathcal{T}_V^{(1)}$ is represented by at most $\mathcal{O}(n)$ integers.
For these reasons, the input size of the solid torus core recognition problem is measured by the size of an input triangulation of the solid torus.

Let $\mathcal{T}_M$ be a triangulation of a compact $3$-manifold $M$ containing a knot $K$ in its $1$-skeleton $\mathcal{T}_M^{(1)}$.
Let $\mathcal{T}_M''$ denote the triangulation obtained by barycentrically subdividing $\mathcal{T}_M$ twice.
A triangulation of the exterior of $K$ is obtained from $\mathcal{T}_M''$ by removing the tetrahedra containing $K$ in its edges.
From this construction, we have the following.
\begin{lem}\label{lem:exterior}
    Let $\mathcal{T}_M$ be an $n$-tetrahedra triangulation of a compact $3$-manifold $M$ containing a knot $K$ in its $1$-skeleton $\mathcal{T}_M^{(1)}$.
    Then there is a $\mathcal{O}(n)$ time algorithm that, given $\mathcal{T}_M$ and $K$, outputs a triangulation $\mathcal{T}_E$ of the exterior of $K$.
    Moreover, $\text{size}(\mathcal{T}_E)$ is at most $\mathcal{O}(n)$.
\end{lem}
\begin{proof}
    The barycentric subdivision is performed in $\mathcal{O}(\text{size}(\mathcal{T}_M)) = \mathcal{O}(n)$ time and multiplies the number of tetrahedra by 24.
    Thus, the triangulation $\mathcal{T}_M''$ obtained by barycentrically subdividing $\mathcal{T}_M$ twice is obtained in $\mathcal{O}(n)$ time,
    and we have $\text{size}(\mathcal{T}_M'') = 24^2n$.
    Therefore, we can obtain a triangulation $\mathcal{T}_E$ of the exterior of $K$ by removing the tetrahedra containing $K$ in $\mathcal{O}(\text{size}(\mathcal{T}_M'')) = \mathcal{O}(n)$ time,
    and the number of tetrahedra of $\mathcal{T}_E$ is less than $24^2n$. 
\end{proof}

\subsection{An algorithm for calculating the rotation number of a knot in the solid torus}
Suppose that $\mathcal{T}_V$ is a triangulation of the solid torus $V$ and $K$ is a knot in $V$ represented by a collection of edges of $\mathcal{T}_V^{(1)}$.
In this subsection, we describe that $|r(K)|$ is calculated in polynomial time of $\text{size}(\mathcal{T}_V)$.
\begin{lem}\label{lem:rotationNum}
    Let $\mathcal{T}_V$ be an $n$-tetrahedra triangulation of the solid torus $V$ and $K$ be a knot in $V$ represented by a collection of edges of $\mathcal{T}_V^{(1)}$.
    Then there is an algorithm that, given $\mathcal{T}_V$ and $K$, outputs $|r(K)|$ in polynomial time of $n$. 
\end{lem}
\begin{proof}
    For each dimension $k \geq 0$, we denote the $k$-simplices of $\mathcal{T}_V$ by $c^k_1, \ldots, c^k_{n_k}$, and fix an orientation for each $k$-simplex $c^k_i$.
    Let $C_k$ denote the $k$-chain group of $\mathcal{T}_V$ over $\mathbb{Z}$.
    Let $D_k \in \mathbb{Z}^{n_{k-1} \times n_k}$ denote the representation matrix of the boundary operator $\partial_k: C_k \to C_{k-1}$ with respect to the standard basis.
    Then there are unimodular matrices $P \in \mathbb{Z}^{n_{0} \times n_{0}}$ and $Q \in \mathbb{Z}^{n_{1} \times n_{1}}$ such that $PD_1Q$ is the Smith normal form of $D_1$.
    The matrices $P$ and $Q$ can be calculated in polynomial time of $n$ (\cite{S}).
    Let $Q = (\bm{q}_1, \ldots, \bm{q}_{n_1})$, $D_2 = (\bm{d}_1, \ldots, \bm{d}_{n_2})$, and $r_k = \text{rank}(\text{Im}(D_k))$ for each $k$.
    We see that $\{\bm{q}_{1}, \ldots, \bm{q}_{n_1}\}$ is a basis of $C_1$ and $\{\bm{q}_{r_1+1}, \ldots, \bm{q}_{n_1}\}$ is a basis of $\text{Ker}(D_1)$.
    
    A basis of $\text{Im}(D_2)$ is obtained as follows.
    Let $S = \emptyset$.
    For each $i$ $(1 \leq i \leq n_2)$, if $S \cup \{\bm{d}_{i}\}$ is linearly independent, then add $\bm{d}_{i}$ to $S$.
    We can check whether a set of vectors $\{\bm{d}_{j_1}, \ldots, \bm{d}_{j_m}\}$ is linearly independent by calculating the Smith normal form of the matrix $(\bm{d}_{j_1}, \ldots, \bm{d}_{j_m})$ and checking the number of elementary divisors is $m$.
    Since the Smith normal form of the matrix $(\bm{d}_{j_1}, \ldots, \bm{d}_{j_m})$ is calculated in polynomial time of $n$, a basis $\{\bm{d}_{i_1}, \ldots, \bm{d}_{i_{r_2}}\}$ of $\text{Im}(D_2)$ is obtained in polynomial time of $n$.
    
    Since $H_1(V; \mathbb{Z}) \simeq \mathbb{Z}$ , there is a vector $\bm{q}_j$ $(r_1+1 \leq j \leq n_1)$ such that $\{\bm{q}_j, \bm{d}_{i_1}, \ldots, \bm{d}_{i_{r_2}}\}$ is a basis of $\text{Ker}(D_1)$.
    We can find $\bm{q}_j$ by calculating $\text{rank}( \langle \bm{q}_{j'}, \bm{d}_{i_1}, \ldots, \bm{d}_{i_{r_2}} \rangle )$ for each $j' \in \{r_1+1, \ldots, n_1\}$.
    The rank of $\langle \bm{q}_{j'}, \bm{d}_{i_1}, \ldots, \bm{d}_{i_{r_2}} \rangle$ is obtained by calculating the Smith normal form of $(\bm{q}_{j'}, \bm{d}_{i_1}, \ldots, \bm{d}_{i_{r_2}})$.
    Thus, a vector $\bm{q}_j$ is obtained in polynomial time of $n$.
    Then we see that $\{\bm{q}_j, \bm{d}_{i_1}, \ldots, \bm{d}_{i_{r_2}} ,\bm{q}_1, \ldots, \bm{q}_{r_1}\}$ is a basis of $C_1$.
    Let $X = ( \bm{q}_j,\bm{d}_{i_1}, \ldots, \bm{d}_{i_{r_2}}, \bm{q}_1, \ldots, \bm{q}_{r_1})$.
    Now, for each $1$-cycle $\alpha$ in $V$, if $\alpha = a_1c^1_1 + \cdots + a_{n_1}c^1_{n_1}$, then $[\alpha]$ is the first element of $X^{-1}(a_1, \ldots, a_{n_1})^\top$. 
    The inverse of $X$ is calculated in polynomial time of $n$ by using the the Gaussian elimination method.
    Thus, given $\mathcal{T}_V$ and $K$, we can calculate $|r(K)|$ in polynomial time of $n$.
\end{proof}

\subsection{Normal surfaces}
A properly embedded arc in a triangle is an \textit{elementary arc} if the arc connects the interior of distinct edges of the triangle.
An \textit{elementary disk} in a tetrahedron $\Delta_i$ is a properly embedded disk in $\Delta_i$ whose boundary consists of three or four elementary arcs of the faces of $\Delta_i$ depicted as in Figure \ref{fig:elementaryDisk}.
Two elementary disks in a tetrahedron are said to be of the \textit{same type} if the vertices of them are on the same edges of the tetrahedron.
There are seven types of elementary disks in a tetrahedron.
A properly embedded surface $F$ in a compact $3$-manifold $M$ with a triangulation $\mathcal{T}_M$ is called a \textit{normal surface} with respect to $\mathcal{T}_M$
if for each tetrahedron $\Delta_i$ of $\mathcal{T}_M$, $\Delta_i \cap F$ is a collection of disjoint elementary disks.
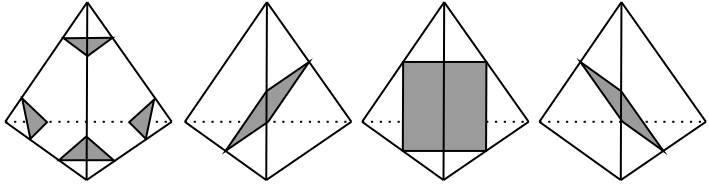
\begin{figure}[htbp]
    \centering
    \tikzset{every picture/.style={line width=0.75pt}} 

\begin{tikzpicture}[x=0.75pt,y=0.75pt,yscale=-0.6,xscale=0.6]

\draw    (80,11) -- (11,111.67) ;
\draw    (80,11) -- (151,111.67) ;
\draw    (79.43,160.71) -- (151,111.67) ;
\draw    (11,111.67) -- (79.43,160.71) ;
\draw  [dash pattern={on 0.84pt off 2.51pt}]  (11,111.67) -- (151,111.67) ;
\draw  [fill={rgb, 255:red, 155; green, 155; blue, 155 }  ,fill opacity=1 ] (60.2,41.4) -- (80.2,56.4) -- (101.2,41) -- cycle ;
\draw  [fill={rgb, 255:red, 155; green, 155; blue, 155 }  ,fill opacity=1 ] (56.2,144) -- (79.2,124) -- (102.2,144) -- cycle ;
\draw  [fill={rgb, 255:red, 155; green, 155; blue, 155 }  ,fill opacity=1 ] (46.2,112) -- (25.2,92) -- (32.2,126) -- cycle ;
\draw  [fill={rgb, 255:red, 155; green, 155; blue, 155 }  ,fill opacity=1 ] (115.2,112) -- (136.2,93) -- (129.2,126) -- cycle ;
\draw    (231,11) -- (162,111.67) ;
\draw    (231,11) -- (302,111.67) ;
\draw    (230.43,160.71) -- (302,111.67) ;
\draw    (162,111.67) -- (230.43,160.71) ;
\draw  [dash pattern={on 0.84pt off 2.51pt}]  (162,111.67) -- (302,111.67) ;
\draw    (380,11) -- (311,111.67) ;
\draw    (380,11) -- (451,111.67) ;
\draw    (379.43,160.71) -- (451,111.67) ;
\draw    (311,111.67) -- (379.43,160.71) ;
\draw  [dash pattern={on 0.84pt off 2.51pt}]  (311,111.67) -- (451,111.67) ;
\draw    (529,11) -- (460,111.67) ;
\draw    (529,11) -- (600,111.67) ;
\draw    (528.43,160.71) -- (600,111.67) ;
\draw    (460,111.67) -- (528.43,160.71) ;
\draw  [dash pattern={on 0.84pt off 2.51pt}]  (460,111.67) -- (600,111.67) ;
\draw  [fill={rgb, 255:red, 155; green, 155; blue, 155 }  ,fill opacity=1 ] (232,111.67) -- (196.21,136.19) -- (230.71,85.86) -- (266.5,61.33) -- cycle ;
\draw  [fill={rgb, 255:red, 155; green, 155; blue, 155 }  ,fill opacity=1 ] (564.21,136.19) -- (530,111.67) -- (494.5,61.33) -- (528.71,85.86) -- cycle ;
\draw  [fill={rgb, 255:red, 155; green, 155; blue, 155 }  ,fill opacity=1 ] (415.21,136.19) -- (345.21,136.19) -- (345.5,61.33) -- (415.5,61.33) -- cycle ;
\draw    (80,11) -- (79.43,160.71) ;
\draw    (231,11) -- (230.43,160.71) ;
\draw    (380,11) -- (379.43,160.71) ;
\draw    (529,11) -- (528.43,160.71) ;

\end{tikzpicture}
    \caption{Elementary disks}
    \label{fig:elementaryDisk}
\end{figure}

Let $n$ be the size of a triangulation $\mathcal{T}_M$ of a compact $3$-manifold $M$.
We record a normal surface $F$ with respect to $\mathcal{T}_M$ as the vector $v(F) \in \mathbb{Z}^{7n}$, where each coordinate describes the number of elementary disks of each type in each tetrahedron.
The vector $v(F)$ is called the \textit{vector representation} of a normal surface $F$.

A vector in $\mathbb{Z}^{7n}$ does not always represent a normal surface with respect to an $n$-tetrahedra triangulation $\mathcal{T}_M$.
We describe the conditions for a vector $\bm{x} = (x_{1,1}, \ldots, x_{1,7},$
$ x_{2,1}, \ldots, x_{n,7}) \in \mathbb{Z}^{7n}$ to represent a normal surface $F$.
Firstly, each coordinate $x_{i,j}$ is greater than or equal to $0$.
This condition is called the \textit{non-negative condition}.
Secondly, the elementary disks in two adjacent tetrahedra are glued together.
Since for each face of a tetrahedron of $\mathcal{T}_M$, there are two types of elementary disks whose intersection with the face are the same type elementary arcs,
the equation
\[
    x_{i, s} + x_{i, t} = x_{j, u} + x_{j, w}
\]
holds for each type of elementary arcs of an interior face of $\mathcal{T}_M$ (See Figure \ref{fig:matching}).
Since there are three types of elementary arcs in each interior face and at most $2n$ interior faces in $\mathcal{T}_M$, there are at most $6n$ equations.
The matrix $A_{\mathcal{T}_M}$ is defined by the coefficient matrix of these equations.
We call this matrix the \textit{matching matrix} of $\mathcal{T}_M$.
The \textit{matching condition} is the condition that $A_{\mathcal{T}_M}\bm{x} = \bm{0}$.
\begin{figure}[htpb]
    \centering
    \tikzset{every picture/.style={line width=0.75pt}} 

\begin{tikzpicture}[x=0.75pt,y=0.75pt,yscale=-0.7,xscale=0.7]

\draw    (110,7.34) -- (9,91.67) ;
\draw    (110,7.34) -- (149,91.67) ;
\draw    (77.43,140.71) -- (149,91.67) ;
\draw    (9,91.67) -- (77.43,140.71) ;
\draw  [dash pattern={on 0.84pt off 2.51pt}]  (9,91.67) -- (149,91.67) ;
\draw    (185,9.34) -- (158,91.67) ;
\draw    (185,9.34) -- (298,91.67) ;
\draw    (226.43,140.71) -- (298,91.67) ;
\draw    (158,91.67) -- (226.43,140.71) ;
\draw  [dash pattern={on 0.84pt off 2.51pt}]  (158,91.67) -- (298,91.67) ;
\draw  [fill={rgb, 255:red, 155; green, 155; blue, 155 }  ,fill opacity=1 ] (98,90.34) -- (53,123.34) -- (93.71,74.03) -- (129.5,49.51) -- cycle ;
\draw  [fill={rgb, 255:red, 155; green, 155; blue, 155 }  ,fill opacity=1 ] (71,40.34) -- (99,54.34) -- (124,40.63) -- cycle ;
\draw  [fill={rgb, 255:red, 155; green, 155; blue, 155 }  ,fill opacity=1 ] (262.21,116.19) -- (214,90.34) -- (174.5,41.51) -- (201,58.34) -- cycle ;
\draw  [fill={rgb, 255:red, 155; green, 155; blue, 155 }  ,fill opacity=1 ] (178,33.03) -- (196,43.34) -- (218,32.63) -- cycle ;
\draw  [fill={rgb, 255:red, 155; green, 155; blue, 155 }  ,fill opacity=1 ] (80,34.03) -- (101,44.34) -- (122,33.63) -- cycle ;
\draw    (110,7.34) -- (77.43,140.71) ;
\draw  [fill={rgb, 255:red, 155; green, 155; blue, 155 }  ,fill opacity=1 ] (252,122.34) -- (200,90.34) -- (171.5,50.51) -- (205.71,75.03) -- cycle ;
\draw    (185,9.34) -- (226.43,140.71) ;

\end{tikzpicture}
    \caption{Elementary disks in adjacent tetrahedra}
    \label{fig:matching}
\end{figure}
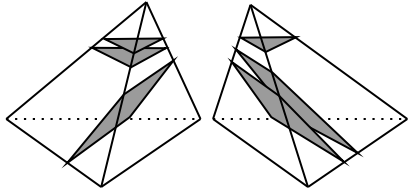
If there are distinct types of quadrilateral elementary disks in a tetrahedron, they must intersect.
The final condition is that each tetrahedron has at most one type quadrilateral elementary disks.
This is called the \textit{quadrilateral condition.}
Haken showed in \cite{Haken} that 
a vector $\bm{x} \in \mathbb{Z}^{7n}$ represents a normal surface with respect to an $n$-tetrahedra triangulation $\mathcal{T}_M$ if and only if 
$\bm{x}$ satisfies the non-negative condition, the matching condition, and the quadrilateral condition.

Vertex surfaces are introduced by Jaco and Ortel \cite{JO} and by Jaco and Tollefson \cite{JT}.
Let $\mathcal{T}_M$ be an $n$-tetrahedra triangulation of a compact $3$-manifold $M$ and $A_{\mathcal{T}_M}$ denote the matching matrix of ${\mathcal{T}_M}$.
The \textit{Haken normal cone} $\mathscr{C}_{\mathcal{T}_M}$ of $\mathcal{T}_M$ is the polyhedral cone in $\mathbb{R}^{7n}$ defined by $A_{\mathcal{T}_M}\bm{x} = 0$ and $x_i \geq 0$ for each coordinate.
The integer points of $\mathscr{C}_{\mathcal{T}_M}$ that satisfy the quadrilateral condition represent the normal surfaces with respect to $\mathcal{T}_M$.
A normal surface $F$ with respect to $\mathcal{T}_M$ is a \textit{vertex surface} if $F$ is connected and $2$-sided in $M$, and the vector representation $v(F)$ is on an extreme ray, namely a $1$-dimensional face, of $\mathscr{C}_{\mathcal{T}_M}$.

Let $M$ be a compact $3$-manifold.
A properly embedded surface $F$ in $M$ that is not the disk or the $2$-sphere is said to be \textit{essential} if $F$ is incompressible, $\partial$-incompressible, and not parallel to $\partial M$.
\begin{thm}[Jaco-Tollefson \cite{JT}]
    Let $\mathcal{T}_M$ be a triangulation of $M = S \times [0,1]$, where $S$ is a closed surface that is not a $2$-sphere or a projective plane.
    Then there is an essential two-sided annulus $F$ that is a vertex surface with respect to $\mathcal{T}_M$.
\end{thm}

If $K$ is the core of the solid torus, then the exterior of $K$ is homeomorphic to $T^2 \times [0,1]$, where $T^2$ is the torus.
Therefore, we have the following lemma.
\begin{lem}\label{lem:vertexEssentialAnnulus}
    Let $K$ be the core of the solid torus $V$ and $E = V - \text{int}N(K)$.
    Assume that $\mathcal{T}_E$ is a triangulation of $E$.
    Then there is an essential annulus $F$ that is a vertex surface with respect to $\mathcal{T}_E$.
\end{lem}

Let $M$ be a compact irreducible $3$-manifold.
A collection of properly embedded disjoint disks $\{D_1, \ldots, D_n\}$ in $M$ is called a \textit{complete disk system} for $M$ if each boundary component of the $3$-manifold obtained by cutting $M$ along $\bigcup_{i=1}^n D_i$ is incompressible.
\begin{thm}[Jaco-Tollefson \cite{JT}]\label{thm:compDisk}
    Let $\mathcal{T}_M$ be a triangulation of a compact irreducible $3$-manifold $M$ whose boundary is compressible.
    Then there is a complete disk system $\{D_1, \ldots, D_n\}$ for $M$ such that each disk $D_i$ is a vertex surface with respect to $\mathcal{T}_M$.
\end{thm}

Vertex surfaces play an important role in analyzing the computational complexity of algorithms using normal surfaces.
Hass, Lagarias, and Pippenger showed the following.
\begin{thm}[Hass-Lagarias-Pippenger \cite{HLP}]\label{thm:vertexBound}
    Let $\mathcal{T}_M$ be an $n$-tetrahedra triangulation of a compact $3$-manifold $M$.
    Assume that $F$ is a vertex surface with respect to $\mathcal{T}_M$ represented by $v(F) = (x_1, \ldots, x_{7n}) \in \mathbb{Z}^{7n}$.
    Then each $x_i$ is bounded from above by $2^{7n-1}$.
\end{thm}
This theorem implies that any vertex surface with respect to a triangulation $\mathcal{T}_M$ of a compact $3$-manifold $M$ is represented by a binary string whose length is at most $\mathcal{O}(\text{size}(\mathcal{T}_M)^2)$.
Thus, if $\mathcal{T}_M$ is an $n$-tetrahedra triangulation, then we can guess a vertex surface with respect to $\mathcal{T}_M$ in non-deterministic polynomial time of $n$.
Indeed, the unknot recognition problem is solved in non-deterministic polynomial time by guessing a vertex surface with respect to a triangulation of the exterior of a given knot.
\begin{thm}[Hass-Lagarias-Pippenger \cite{HLP}]\label{thm:unknot}
    There is a non-deterministic polynomial time algorithm that, given a knot diagram $D$, decides whether the knot represented by $D$ is the unknot. 
\end{thm}

Assume that $\bm{x} = (x_1, \ldots, x_{7n}) \in \mathbb{Z}^{7n}$ is a vector such that $x_i \leq 2^{7n-1}$ for each coordinate, 
that is, $\bm{x}$ is a candidate for the vector representation of a vertex surface with respect to an $n$-tetrahedra triangulation of a compact $3$-manifold.
In this situation, some computations on normal surfaces can be performed in polynomial time.
\begin{lem}\label{lem:normal_alg}
    There is a polynomial time algorithm that, given an $n$-tetrahedra triangulation $\mathcal{T}_M$ of a compact $3$-manifold $M$ and a vector $\bm{x} = (x_1, \ldots, x_{7n}) \in \mathbb{Z}^{7n}$ such that $x_i \leq 2^{7n-1}$ for each coordinate, 
    decides whether $\bm{x}$ represents a normal surface with respect to $\mathcal{T}_M$.
\end{lem}
\begin{proof}
    We can decide whether $\bm{x}$ represents a normal surface with respect to $\mathcal{T}_M$ by verifying the non-negative condition, the matching condition, and the quadrilateral condition.
    Since each coordinate is less than or equal to $2^{7n-1}$, we can verify each condition in polynomial time of $n$.
\end{proof}

We describe a polynomial time algorithm that calculates the Euler characteristic $\chi(F)$ of a normal surface $F$.
Schleimer \cite{Schleimer} constructed a polynomial time algorithm that calculates $\chi(F)$ in the case where $F$ is closed.
By adding a slight change to this algorithm, we can calculate $\chi(F)$ of a normal surface $F$ with non-empty boundary.
\begin{lem}\label{lem:euler}
    Let $\mathcal{T}_M$ be an $n$-tetrahedra triangulation of a compact $3$-manifold $M$ (possibly $\partial M \neq \emptyset$) and $F$ be a normal surface with respect to $\mathcal{T}_M$ represented by a vector $\bm{x} = (x_1, \ldots, x_{7n})\in \mathbb{Z}^{7n}$.
    Assume that $x_{i} \leq 2^{7n-1}$ for each coordinate.
    Then there is a polynomial time algorithm that, given $\mathcal{T}_M$ and  $\bm{x}$, calculates the Euler characteristic $\chi(F)$.
\end{lem}
\begin{proof}
    The Euler characteristic $\chi(F)$ is calculated by the formula $n_f-n_e+n_v$, where $n_f$, $n_e$, and $n_v$ is the number of faces, edges, and vertices of $F$, respectively.
    The number of faces $n_f$ is the sum of the coordinates $\sum x_i$.
    The number of edges $n_e$ is calculated as follows.
    For each face $f$ of $\mathcal{T}_M^{(2)}$ and each integer $i \in \{1, \ldots, 7n\}$, set $\epsilon_{f, i} = 1$ if the elementary disks described by $x_i$ meet $f$, otherwise set $\epsilon_{f, i} = 0$.
    Then \[n_e = \sum_{f:\text{face}} \sum_{i=1}^{7n} \frac{\epsilon_{f,i}x_i}{\text{deg}(f)},\]
    where $\text{deg}(f)$ denotes the number of tetrahedra of $\mathcal{T}_M$ containing $f$. 
    Similarly, the number of vertices $n_v$ is also calculated as follows.
    For each edge $e$ of $\mathcal{T}_M^{(1)}$ and each integer $i \in \{1, \ldots, 7n\}$, set $\epsilon_{e, i} = 1$ if the elementary disks described by $x_i$ meet $e$, otherwise set $\epsilon_{e, i} = 0$.
    Then \[n_v = \sum_{e:\text{edge}} \sum_{i=1}^{7n} \frac{\epsilon_{e,i}x_i}{\text{deg}(e)},\]
    where $\text{deg}(e)$ denotes the number of tetrahedra of $\mathcal{T}_M$ containing $e$.
    Since each coordinate $x_i$ is less than or equal to $2^{7n-1}$, we can calculate these values in polynomial time of $n$.
\end{proof}

\section{Knots in the solid torus}
In this section, we consider a necessary and sufficient condition for a knot $K$ in the solid torus $V$ to be the core of $V$.
The aim of this section is to prove Lemma \ref{lem:condOfCore}.
\begin{lem}\label{lem:condOfCore}
    Let $K$ be a knot in a solid torus $V$, and assume that $|r(K)| = 1$.
    Let $W$ be a solid torus with the meridian $m_W$ and $\alpha$ be an essential simple closed curve in $\partial V$ such that $\alpha$ is not the meridian of $V$.
    Suppose that $M$ is the $3$-manifold obtained by gluing $\partial V$ and $\partial W$ so that $\alpha$ and $m_W$ are identified and $K_M$ is the knot in $M$ obtained from $K$.
    Then $K$ is the core of $V$ if and only if there is a properly embedded essential annulus $A$ in $E_V = V - \text{int}N(K)$ such that $\partial A$ meets both $\partial V$ and $\partial N(K)$,
    and there is a properly embedded essential disk $D$ in $E_M = M - \text{int}N(K_M)$.
\end{lem}
Note that even though there is a properly embedded essential disk in $E_M$, $K$ may not necessarily be the core of $V$ in the situation of Lemma \ref{lem:condOfCore}.
For example, we consider the knot $K$ in the solid torus $V$ depicted as in Figure \ref{fig:ex1}, and suppose that $\alpha$ is a longitude of $V$.
In this case, $M$ is the $3$-sphere $\mathbb{S}^3$, and $K_M$ is the unknot in $\mathbb{S}^3$.
Thus, there is a properly embedded essential disk in $E_M$, but $K$ is not the core of $V$.

\begin{lem}\label{lem:irr}
    Let $K$ be a knot in the solid torus $V$, and assume that $|r(K)| \neq 0$.
    Then the exterior $E$ of $K$ is irreducible and $\partial$-irreducible.
\end{lem}
\begin{proof}
    Suppose that $E$ is reducible, i.e. there is a properly embedded essential $2$-sphere $S$ in $E$.
    This implies that $|r(K)| = 0$.
    This is a contradiction.
    Next, suppose that $E$ is $\partial$-reducible, i.e. there is a properly embedded essential disk $D$ in $E$.
    This also implies that $|r(K)| = 0$.
    This contradicts that $|r(K)| \neq 0$. 
\end{proof}

The next lemma is used in Lemma \ref{lem:planarSurface} and \ref{lem:condition2}.
\begin{lem}\label{lem:noBigon}
    Let $M$ be a compact $3$-manifold with non-empty boundary.
    Assume that $F_1$ and $F_2$ are properly embedded surfaces in $M$ such that $F_1$ and $F_2$ intersect transversely.
    Then we can isotope $F_1$ and $F_2$ so that there are no bigons in $\partial M$ whose boundaries consist of parts of $\partial F_1$ and $\partial F_2$.
    Moreover, this procedure does not increase the number of intersections of $F_1$ and $F_2$.
\end{lem}
\begin{proof}
    Suppose that there is a bigon $B$ whose boundary consists of two arcs $\alpha_1$ and $\alpha_2$, where $\alpha_i$ is a sub-arc of $\partial F_i$ for each $i$.
    Let $c_1$ and $c_2$ denote the vertices of $B$ and $\beta_i$ denote the arc of $F_1 \cap F_2$ containing $c_i$ for each $i$.
    In this situation, we can remove $B$ by moving $\alpha_1$ to $\alpha_2$ along $B$ (See Figure \ref{fig:isotopy}).
    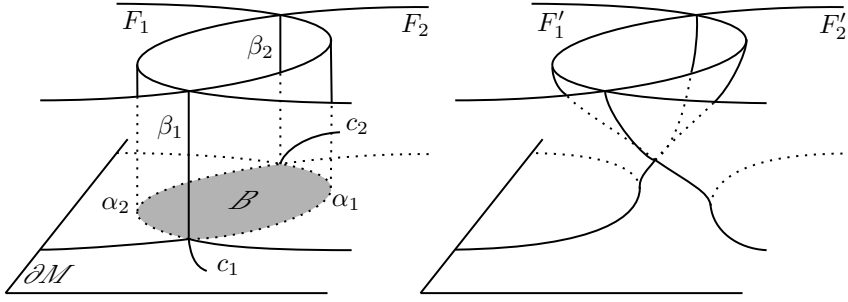
\begin{figure}[htbp]
        \centering
        \tikzset{every picture/.style={line width=0.75pt}} 

\begin{tikzpicture}[x=0.75pt,y=0.75pt,yscale=-0.7,xscale=0.7]

\draw    (3,222.34) -- (235,222.34) ;
\draw    (3,222.34) -- (91,111.34) ;
\draw    (28,83.34) .. controls (234,83.69) and (348,14.69) .. (83,13.69) ;
\draw    (238,85.34) -- (237.36,38.71) ;
\draw    (98,81.34) -- (97.87,57.46) ;
\draw    (307.87,16.11) .. controls (101.88,14.26) and (-12.62,82.42) .. (252.36,85.37) ;
\draw  [dash pattern={on 0.84pt off 2.51pt}]  (307.87,122.11) .. controls (214,121.34) and (95,142.34) .. (97.87,165.11) ;
\draw    (135,183.34) .. controls (152,187.34) and (174,191.34) .. (252.36,191.37) ;
\draw  [dash pattern={on 0.84pt off 2.51pt}]  (83,121.34) .. controls (125,121.34) and (237,124.34) .. (237.36,146.37) ;
\draw    (135,183.34) .. controls (122,185.34) and (52,191.34) .. (28,191) ;
\draw    (135,183.34) -- (135,77.34) ;
\draw  [dash pattern={on 0.84pt off 2.51pt}]  (135,183.34) .. controls (177,183.34) and (240,163.34) .. (237.36,146.37) ;
\draw  [dash pattern={on 0.84pt off 2.51pt}]  (135,183.34) .. controls (126,181.34) and (99,174.34) .. (97.87,165.11) ;
\draw    (201,63.34) -- (201,22.34) ;
\draw  [dash pattern={on 0.84pt off 2.51pt}]  (201,129.34) -- (201,63.34) ;
\draw    (302,222.34) -- (534,222.34) ;
\draw    (302,222.34) -- (390,111.34) ;
\draw    (327,83.34) .. controls (533,83.69) and (647,14.69) .. (382,13.69) ;
\draw    (606.87,16.11) .. controls (400.88,14.26) and (286.38,82.42) .. (551.36,85.37) ;
\draw  [dash pattern={on 0.84pt off 2.51pt}]  (460,147.34) .. controls (460,140.34) and (454,122.34) .. (382,121.34) ;
\draw  [draw opacity=0][fill={rgb, 255:red, 0; green, 0; blue, 0 }  ,fill opacity=0.3 ] (135,183.34) .. controls (80,171.34) and (75,147.34) .. (201,129.34) .. controls (295,145.34) and (184,185.34) .. (135,183.34) -- cycle ;
\draw  [dash pattern={on 0.84pt off 2.51pt}]  (97.87,165.11) -- (98,81.34) ;
\draw  [dash pattern={on 0.84pt off 2.51pt}]  (237.36,146.37) -- (238,85.34) ;
\draw    (327,191) .. controls (433,189.34) and (459,166.34) .. (460,147.34) ;
\draw    (551.36,191.37) .. controls (530,190.34) and (513,177.34) .. (511,159.34) ;
\draw    (460,147.34) .. controls (459,138.34) and (466,133.34) .. (471,126.34) ;
\draw    (511,159.34) .. controls (510,149.34) and (496,143.34) .. (471,126.34) ;
\draw  [dash pattern={on 0.84pt off 2.51pt}]  (511,159.34) .. controls (512,146.34) and (525,122.34) .. (606.36,121.71) ;
\draw  [dash pattern={on 0.84pt off 2.51pt}]  (471,126.34) .. controls (455,112.34) and (414,89.34) .. (407,79.34) ;
\draw    (407,79.34) .. controls (402,76.34) and (397,68.34) .. (397,59.34) ;
\draw  [dash pattern={on 0.84pt off 2.51pt}]  (471,126.34) .. controls (487,108.34) and (504,94.34) .. (513,85.34) ;
\draw    (513,85.34) .. controls (525,74.34) and (537,49.34) .. (537,39.34) ;
\draw    (471,126.34) .. controls (465,127.34) and (437,95.34) .. (435,77.34) ;
\draw  [dash pattern={on 0.84pt off 2.51pt}]  (471,126.34) .. controls (486,116.34) and (494,83.34) .. (497,63.34) ;
\draw    (135,183.34) .. controls (136,193.34) and (139,203.34) .. (148,206.34) ;
\draw    (201,129.34) .. controls (203,120.34) and (220,107.34) .. (244,106.34) ;
\draw    (497,63.34) .. controls (500,53.34) and (502,32.34) .. (501,22.34) ;

\draw (171.53,146.4) node [anchor=north west][inner sep=0.75pt]  [xslant=0.73]  {$B$};
\draw (85,17.09) node [anchor=north west][inner sep=0.75pt]    {$F_{1}$};
\draw (286,18.74) node [anchor=north west][inner sep=0.75pt]    {$F_{2}$};
\draw (236,149.4) node [anchor=north west][inner sep=0.75pt]    {$\alpha _{1}$};
\draw (70,151.4) node [anchor=north west][inner sep=0.75pt]    {$\alpha _{2}$};
\draw (110,92.4) node [anchor=north west][inner sep=0.75pt]    {$\beta _{1}$};
\draw (175,34.4) node [anchor=north west][inner sep=0.75pt]    {$\beta _{2}$};
\draw (153,196.4) node [anchor=north west][inner sep=0.75pt]    {$c_{1}$};
\draw (246,94.4) node [anchor=north west][inner sep=0.75pt]    {$c_{2}$};
\draw (21.15,200) node [anchor=north west][inner sep=0.75pt]  [xslant=0.55]  {$\partial M$};
\draw (384,17.09) node [anchor=north west][inner sep=0.75pt]    {$F'_{1}$};
\draw (587,18.4) node [anchor=north west][inner sep=0.75pt]    {$F'_{2}$};

\end{tikzpicture}
        \caption{An isotopy to remove a bigon $B$ in the boundary of $M$}
        \label{fig:isotopy}
    \end{figure}
    If $\beta_1$ and $\beta_2$ are distinct edges, then this isotopy decrease $|F_1 \cap F_2|$.
    Otherwise this isotopy does not change $|F_1 \cap F_2|$.
    By repeating the above procedure, we can remove all bigons in the boundary of $M$ without increasing $|F_1 \cap F_2|$.
\end{proof}

Suppose the same situation of Lemma \ref{lem:condOfCore}.
First, we consider the case where $A$ is a properly embedded annulus in the exterior of a knot $K$ in the solid torus $V$ such that $\partial A \cap \partial N(K)$ is not the meridian of $N(K)$.
\begin{lem}\label{lem:condition2}
    Let $E$ denote the exterior of a knot $K$ in the solid torus $V$.
    Assume that $|r(K)| = 1$.
    Then $K$ is the core of $V$ if and only if 
    there is a properly embedded essential annulus $A$ in $E$ such that $\partial A$ meets both $\partial N(K)$ and $\partial V$, and $\partial A \cap \partial N(K)$ is not the meridian of $N(K)$.
\end{lem}
\begin{proof}
    Suppose that $K$ is the core of $V$.
    Since $K$ is parallel to $\partial V$, there is a properly embedded essential annulus $A$ in $E$ such that $\partial A$ meets both $\partial N(K)$ and $\partial V$, and $\partial A \cap \partial N(K)$ is not the meridian of $N(K)$.

    Conversely, suppose that $A$ is a properly embedded essential annulus such that $\partial A$ meets both $\partial N(K)$ and $\partial V$, and $\partial A \cap \partial N(K)$ is not the meridian of $N(K)$.
    Let $m_k$ denote the meridian of $N(K)$.
    We take $A$ so that $|\partial A \cap m_k|$ is minimal up to isotopy of $A$.
    If $|\partial A \cap m_k| = 1$, then we see that $K$ is parallel to $\partial V$.
    Since $|r(K)| = 1$, $K$ is the core of $V$.
    
    Suppose that $|\partial A \cap m_k| \geq 2$.
    We take the meridian disk $D$ of $V$ so that $|D \cap K|$ is minimal up to isotopy of $D$.
    Let $F$ denote the surface $D - \text{int}(D \cap N(K))$ in $E$.

    We show that $F$ is incompressible and $\partial$-incompressible in $E$.
    Suppose that there is a compression disk $\delta$ for $F$.
    Then $\partial \delta$ divides $F$ into the two sub-surfaces $d_1$ and $d_2$.
    Assume that $d_1$ contains $\partial F \cap \partial V$.
    Since $\partial \delta$ is essential in $F$, we have $|(\delta \cup d_1) \cap \partial N(K)| < |F \cap \partial N(K)|$.
    This contradicts the minimality of $|D \cap K|$.
    Therefore, $F$ is incompressible.

    Next, suppose that there is a $\partial$-compression disk $\delta$ for $F$.
    First, we consider the case where the arc $\partial \delta \cap F$ connects the same component of $\partial F$.
    The arc $\partial \delta \cap F$ divides $F$ into the two sub-surfaces $d_1$ and $d_2$.
    We consider the first homology class $[\partial D] \in H(\partial V, \mathbb{Z})$.
    Since we have $0 \neq [\partial D] = [\partial (d_1 \cup \delta)] + [\partial (d_2 \cup \delta)]$, 
    either of $\partial (d_1 \cup \delta)$ or $\partial (d_2 \cup \delta)$ is essential in $\partial V$.
    Without loss of generality, assume that $\partial (d_1 \cup \delta)$ is essential in $\partial V$.
    Since $\partial \delta \cap F$ is essential in $F$, we have $|(d_1 \cup \delta) \cap \partial N(K)| < |F \cap \partial N(K)|$.
    This contradicts the minimality of $|D \cap K|$.
    Next, we consider the case where $\partial \delta \cap F$ connects distinct components of $\partial F$.
    The loops $\partial F \cap \partial N(K)$ divide $\partial N(K)$ into annuli.
    Let $S$ denote the annulus containing the arc $\partial \delta \cap \partial E$.
    We regard $N(\delta)$ as $\delta \times [0,1]$.
    Then the disk $(S - (S \cap N(\delta))) \cup \delta \times \{0,1\}$ is a compression disk for $F$.
    This contradicts that $F$ is incompressible. 
    Therefore, $F$ is $\partial$-incompressible.

    We take $F$ so that $|F \cap A|$ is minimal up to isotopy of $F$.
    Using Lemma \ref{lem:noBigon}, all bigons in $\partial E$ bounded by $\partial F$ and $\partial A$ are removed while keeping the minimality of $|F \cap A|$.
    We show that there are no loops and arcs of $F \cap A$ that are inessential in $A$.
    Suppose that $\alpha$ is a loop of $F \cap A$ such that $\alpha$ is inessential in $A$ and is an innermost loop in $A$ with respect to $F \cap A$.
    Let $\delta$ denote the disk in $A$ bounded by $\alpha$.
    Since $F$ is incompressible, $\alpha$ bounds a disk $\delta'$ in $F$.
    We see that the $2$-sphere $\delta \cup \delta'$ bounds a $3$-ball $B$ in $E$ since $E$ is irreducible.
    Thus, $\alpha$ is removed from $F \cap A$ by an isotopy of $F$ along $B$.
    This contradicts the minimality of $|F \cap A|$.
    Next, suppose that $\alpha$ is an arc of $F \cap A$ such that $\alpha$ is inessential in $A$ and is an outermost arc in $A$ with respect to $F \cap A$.
    Let $\delta$ denote the disk in $A$ bounded by $\alpha$ and a sub-arc of $\partial A$.
    Since $F$ is $\partial$-incompressible, $\alpha$ co-bounds a disk $\delta'$ in $F$ with a sub-arc of $\partial F$.
    Since $E$ is $\partial$-irreducible, $\partial \delta \cup \partial \delta'$ bounds a disk $\delta''$ in $\partial F$.
    We see that the $2$-sphere $\delta \cup \delta' \cup \delta''$ bounds a $3$-ball $B$ in $E$ since $E$ is irreducible.
    Therefore, $\alpha$ is removed from $F \cap A$ by an isotopy of $F$ along $B$.
    This is a contradiction.
    Thus, there are no loops and arcs of $F \cap A$ that are inessential in $A$.

    Since each component of $\partial F \cap \partial N(K)$ is the meridian of $N(K)$ and $\partial A \cap \partial N(K)$ is an essential loop that is not the meridian of $N(K)$, 
    we see that $\partial A \cap \partial F \neq \emptyset$.
    This implies that each component of $F \cap A$ is an essential arc in $A$.

    Since there are no bigons in $\partial E$ bounded by $\partial F$ and $\partial A$, 
    there is an integer $m$ such that each component of $\partial F \cap \partial N(K)$ intersects $\partial A \cap \partial N(K)$ as $m$ points.
    By the assumption that $|\partial A \cap m_k| \geq 2$, we see that $m \geq 2$.

    Let $n$ denote the number of components of $\partial F \cap \partial N(K)$, and we show that $n = 1$.
    Suppose that $n \geq 2$.
    Let $(\mathbb{B}^3, \tau)$ denote the tangle in the $3$-ball $\mathbb{B}^3$ such that $(\mathbb{B}^3, \tau)$ is obtained by cutting $(V, K)$ along $D$.
    We see that $\tau$ has $n$ strings $t_1, \ldots, t_n$.
    Let $E_\tau$ denote the exterior of $\tau$, and let $F^+$ and $F^-$ denote the two sub-surfaces of $\partial E_\tau$ obtained from $F$.
    The surface $F$ divides $A$ into $nm$ disks since each component of $F \cap A$ is an essential arc in $A$.
    These disks are denoted by $A_{i, j}$ $(1 \leq i \leq n, 1 \leq j \leq m)$, where $A_{i, j}$ is a disk intersecting $\partial N(t_i)$.
    Now, we have the following:
    \begin{itemize}
        \item for each string $t_i$, there is a disk $A_{i, j}$ since $\partial A \cap \partial N(K)$ is not the meridian of $N(K)$, and 
        \item for each disk $A_{i, j}$, $\partial N(t_i) \cap \partial A_{i, j}$ is an arc that connects $F^+$ and $F^-$ 
                since there are no bigons in $\partial E$ bounded by $\partial F$ and $\partial A$.
    \end{itemize}
    These imply that $(E_\tau, \bigcup A_{i,j})$ is homeomorphic to $(F^+ \times [0,1], (\bigcup A_{i,j} \cap F^+) \times [0,1])$ as a pair.

    Let $a^{\pm}_{i,j}$ denote the arc $A_{i,j} \cap F^{\pm}$ and $b^{\pm}_i$ denote $\partial N(t_i) \cap \partial F^{\pm}$.
    We define $f:F^- \to F^+$ to be the homeomorphism such that $E$ is obtained from $E_\tau$ by gluing $F^+$ and $F^-$ by $f$.
    We also define $g:F^+ \to F^-$ to be the homeomorphism such that $g(b^+_i) = b^-_i$ and $g(a^+_{i,j}) = a^-_{i,j}$ for each $i$ and $j$.
    Note that $g$ is just the projection from $F^+$ to $F^-$ since $(E_\tau, \bigcup A_{i,j}) \simeq (F^+ \times [0,1], (\bigcup A_{i,j} \cap F^+) \times [0,1])$.
    Let $h = f \circ g$.
    If $k > 0$, then the function $h^k:F^+ \to F^+$ is defined by $h^{k-1} \circ h$, and if $k = 0$, then $h^0: F^+ \to F^+$ is defined by the identity map.
    Since $A$ is connected, we have the following:
    \begin{itemize}
        \item $h^{nm}(a^+_{1,1}) = a^+_{1,1}$ and 
        \item for each $i$ and $j$ $(1 \leq i \leq n, 1 \leq j \leq m)$, there is an integer $k$ $(0 \leq k \leq nm-1)$ such that $a^+_{i,j} = h^k(a^+_{1,1})$.
    \end{itemize}
    Let $s$ denote the boundary component $\partial F^+ - \bigcup b^+_i$.
    Now, the points $s \cap \bigcup a^+_{i,j}$ divide $s$ into $nm$ sub-arcs $s_1, \ldots, s_{nm}$.
    We see that for each sub-arc $s_i$, there is an integer $k$ $(0 \leq k \leq nm-1)$ such that $s_i = h^k(s_1)$.

    Next, we show that for each boundary component $b_i^+$, the arcs of $A \cap F$ that meet $b_i^+$ are mutually parallel in $F^+$.
    If $n = 2$, i.e. $F^+$ is the two-punctured disk, then the arcs of $A \cap F$ that meets $b_i^+$ are must mutually parallel.
    Suppose that $n \geq 3$.
    For simplicity, we show that the arcs meeting $b_1^+$ are mutually parallel.
    Suppose that $a_{1,j}^+$ and $a_{1,k}^+$ are not parallel.
    In this situation, $a_{1,j}^+ \cup a_{1,k}^+$ divides the set $\{b_2^+, \ldots, b_n^+\}$ into the two set $X = \{b_2^+, \ldots, b_l^+\}$ and $Y = \{b_{l+1}^+, \ldots, b_n^+\}$.
    Without loss of generality, assume that $|X| \leq |Y|$.
    Let $p$ be an integer such that $h^p(b_1^+)$ is in $X$.
    Then  $h^p(a_{1,j}^+) \cup h^p(a_{1,k}^+)$ divides the set $\{h^p(b_2^+), \ldots, h^p(b_n^+)\}$ into the two set $X^p = \{h^p(b_2^+), \ldots, h^p(b_l^+)\}$ and $Y^p = \{h^p(b_{l+1}^+), \ldots, h^p(b_n^+)\}$.
    Here, we see that $a_{1,j}^+ \cup a_{1,k}^+$ intersects $h^p(a_{1,j}^+) \cup h^p(a_{1,k}^+)$ since $|X - \{h^p(b_1)\}| < |X^p|$ depicted as in Figure \ref{fig:intersectionArcs}.
    This contradicts that $A$ is embedded annulus.
    Thus, we see that the arcs of $A \cap F$ that meet $b_i^+$ are mutually parallel for each $b_i^+$.
    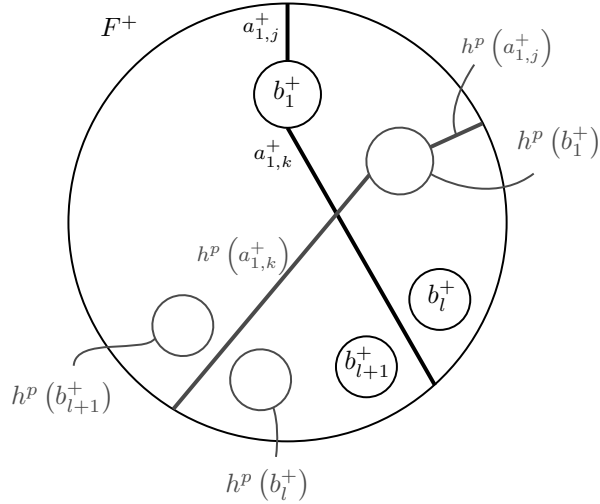
\begin{figure}[htbp]
        \centering
        \tikzset{every picture/.style={line width=0.75pt}} 

\begin{tikzpicture}[x=0.75pt,y=0.75pt,yscale=-1,xscale=1]

\draw   (47.33,119.5) .. controls (47.33,58.47) and (96.81,9) .. (157.83,9) .. controls (218.86,9) and (268.33,58.47) .. (268.33,119.5) .. controls (268.33,180.53) and (218.86,230) .. (157.83,230) .. controls (96.81,230) and (47.33,180.53) .. (47.33,119.5) -- cycle ;
\draw   (140.96,54.87) .. controls (140.96,45.55) and (148.51,38) .. (157.83,38) .. controls (167.15,38) and (174.7,45.55) .. (174.7,54.87) .. controls (174.7,64.19) and (167.15,71.74) .. (157.83,71.74) .. controls (148.51,71.74) and (140.96,64.19) .. (140.96,54.87) -- cycle ;
\draw [line width=1.5]    (157.83,9) -- (157.83,38) ;
\draw [line width=1.5]    (157.83,71.74) -- (232.2,201.74) ;
\draw   (219.26,158.37) .. controls (219.26,149.88) and (226.14,143) .. (234.63,143) .. controls (243.12,143) and (250,149.88) .. (250,158.37) .. controls (250,166.86) and (243.12,173.74) .. (234.63,173.74) .. controls (226.14,173.74) and (219.26,166.86) .. (219.26,158.37) -- cycle ;
\draw   (182.26,192.37) .. controls (182.26,183.88) and (189.14,177) .. (197.63,177) .. controls (206.12,177) and (213,183.88) .. (213,192.37) .. controls (213,200.86) and (206.12,207.74) .. (197.63,207.74) .. controls (189.14,207.74) and (182.26,200.86) .. (182.26,192.37) -- cycle ;
\draw  [color={rgb, 255:red, 74; green, 74; blue, 74 }  ,draw opacity=1 ] (207.46,73.16) .. controls (215.94,69.29) and (225.95,73.03) .. (229.81,81.51) .. controls (233.67,89.99) and (229.93,99.99) .. (221.46,103.86) .. controls (212.98,107.72) and (202.97,103.98) .. (199.11,95.5) .. controls (195.25,87.03) and (198.99,77.02) .. (207.46,73.16) -- cycle ;
\draw [color={rgb, 255:red, 74; green, 74; blue, 74 }  ,draw opacity=1 ][line width=1.5]    (256.2,69.49) -- (229.81,81.51) ;
\draw [color={rgb, 255:red, 74; green, 74; blue, 74 }  ,draw opacity=1 ][line width=1.5]    (199.11,95.5) -- (100.2,213.74) ;
\draw  [color={rgb, 255:red, 74; green, 74; blue, 74 }  ,draw opacity=1 ] (137.88,184.93) .. controls (145.61,181.41) and (154.72,184.82) .. (158.24,192.54) .. controls (161.76,200.27) and (158.35,209.38) .. (150.63,212.9) .. controls (142.91,216.42) and (133.79,213.01) .. (130.27,205.29) .. controls (126.75,197.57) and (130.16,188.45) .. (137.88,184.93) -- cycle ;
\draw  [color={rgb, 255:red, 74; green, 74; blue, 74 }  ,draw opacity=1 ] (98.65,157.59) .. controls (106.38,154.07) and (115.49,157.48) .. (119.01,165.2) .. controls (122.53,172.92) and (119.12,182.04) .. (111.4,185.56) .. controls (103.68,189.08) and (94.56,185.67) .. (91.04,177.95) .. controls (87.52,170.22) and (90.93,161.11) .. (98.65,157.59) -- cycle ;
\draw [color={rgb, 255:red, 74; green, 74; blue, 74 }  ,draw opacity=1 ]   (230,95) .. controls (242,101.84) and (265,107.84) .. (282,94.84) ;
\draw [color={rgb, 255:red, 74; green, 74; blue, 74 }  ,draw opacity=1 ]   (150.63,212.9) .. controls (153,219.84) and (153,229.84) .. (152,237.84) ;
\draw [color={rgb, 255:red, 74; green, 74; blue, 74 }  ,draw opacity=1 ]   (91.04,177.95) .. controls (86,185.84) and (51,176.84) .. (50,191.84) ;
\draw [color={rgb, 255:red, 74; green, 74; blue, 74 }  ,draw opacity=1 ]   (243.01,75.5) .. controls (239,65.84) and (244,51.84) .. (252,39.84) ;

\draw (62,13.4) node [anchor=north west][inner sep=0.75pt]    {$F^{+}$};
\draw (149,44.4) node [anchor=north west][inner sep=0.75pt]    {$b_{1}^{+}$};
\draw (227,147.4) node [anchor=north west][inner sep=0.75pt]    {$b_{l}^{+}$};
\draw (185,180.4) node [anchor=north west][inner sep=0.75pt]    {$b_{l+1}^{+}$};
\draw (272,69.4) node [anchor=north west][inner sep=0.75pt]  [color={rgb, 255:red, 74; green, 74; blue, 74 }  ,opacity=1 ]  {$h^{p}\left( b_{1}^{+}\right)$};
\draw (125,242.4) node [anchor=north west][inner sep=0.75pt]  [color={rgb, 255:red, 74; green, 74; blue, 74 }  ,opacity=1 ]  {$h^{p}\left( b_{l}^{+}\right)$};
\draw (17,198.4) node [anchor=north west][inner sep=0.75pt]  [color={rgb, 255:red, 74; green, 74; blue, 74 }  ,opacity=1 ]  {$h^{p}\left( b_{l+1}^{+}\right)$};
\draw (144.93,21.19) node  [font=\footnotesize]  {$a_{1,j}^{+}$};
\draw (150.93,85.19) node  [font=\footnotesize]  {$a_{1,k}^{+}$};
\draw (268.93,31.19) node  [font=\footnotesize,color={rgb, 255:red, 74; green, 74; blue, 74 }  ,opacity=1 ]  {$h^{p}\left( a_{1,j}^{+}\right)$};
\draw (135.93,136.19) node  [font=\footnotesize,color={rgb, 255:red, 74; green, 74; blue, 74 }  ,opacity=1 ]  {$h^{p}\left( a_{1,k}^{+}\right)$};

\end{tikzpicture}
        \caption{Arcs that are not parallel in $F^+$}
        \label{fig:intersectionArcs}
    \end{figure}

    From this fact and the assumption that $m \geq 2$, we see that there is a sub-arc $s_i$ of $s$ bounded by endpoints of parallel arcs of $A \cap F$.
    This implies that for each integer $k$, $h^k(s_i)$ is bounded by endpoints of parallel arcs of $A \cap F$, 
    that is, each sub-arc of $s$ is bounded by endpoints of parallel arcs of $A \cap F$.
    On the other hand, since $n \geq 2 $, there is a sub-arc of $s$ bounded by endpoints of non-parallel arcs of $A \cap F$.
    This is a contradiction.
    Therefore, we have $n = 1$, i.e. $F$ is an annulus.
    
    Let $F'$ be a disk that is obtained by dividing $F$ by $\partial A_{1,1}$.
    Now, we see that $F' \cup A_{1,1}$ is an annulus that intersects the meridian $m_k$ of $N(K)$ as a point.
    This implies that $K$ is parallel to $\partial V$.
    Since $|r(K)| = 1$, $K$ is the core of $V$.
\end{proof}

Next, we consider the case where $A$ is a properly embedded annulus in the exterior of a knot $K$ such that $\partial A \cap \partial N(K)$ is the meridian of $N(K)$.
\begin{lem}\label{lem:bijection}
    Let $E$ be the exterior of a knot $K$ in the solid torus $V$.
    Suppose that $|r(K)| = 1$.
    Let $i:\partial N(K) \to E$ and $j:\partial V \to E$ denote the inclusion maps,
    and let $i_*:H_1(\partial N(K); \mathbb{Z}) \to H_1(E; \mathbb{Z})$ and $j_*:H_1(\partial N(K); \mathbb{Z}) \to H_1(E; \mathbb{Z})$ denote the homomorphisms induced by $i$ and $j$, respectively.
    Then, $i_*$ and $j_*$ are bijective.
\end{lem}
\begin{proof}
    First, we show that $i_*$ is injective.
    Let $i'_*: \pi_1(\partial N(K)) \to \pi_1(E)$ denote the homomorphism induced by $i$.
    Since $\partial N(K)$ is the torus, there is an isomorphism $f:\pi_1(\partial N(K)) \to H_1(\partial N(K); \mathbb{Z})$.
    By the Hurewicz theorem, $\pi_1(E) / [\pi_1(E), \pi_1(E)]$ is isomorphic to $H_1(E; \mathbb{Z})$, where $[\pi_1(E), \pi_1(E)]$ is the commutator subgroup of $\pi_1(E)$.
    Let $g'$ denote this isomorphism map.
    We denote the natural homomorphism from $\pi_1(E)$ to $\pi_1(E) / [\pi_1(E), \pi_1(E)]$ by $g''$.
    Let $g = g' \circ g''$.
    Now, we have the following commutative diagram.
    \[
      \begin{tikzcd}
        \pi_1(\partial N(K)) \arrow[r,"i'_*"] \arrow[d, "f"]   &        \pi_1(E) \arrow[d, "g"]\\
        H_1(\partial N(K); \mathbb{Z}) \arrow[r, "i_*"] &    H_1(E; \mathbb{Z})
      \end{tikzcd}  
    \]
    Since $|r(K)| \neq 0$, $\partial N(K)$ is incompressible by Lemma \ref{lem:irr}.
    This implies that $i'_*$ is injective.
    Suppose that $[m]$ and $[l]$ are elements of $\pi_1(\partial N(K))$ such that $[m]$ and $[l]$ generate $\pi_1(\partial N(K))$.
    Let $N$ denote the sub-group of $\pi_1(E)$ such that $N$ is generated by $i'_*([m])$ and $i'_*([l])$.
    We see that $N$ is an Abelian group.
    Therefore, $g|_N$ is injective.
    Since $f$ is bijective, we see that $i_* = f^{-1} \circ i'_* \circ g|_N$ is injective.

    Next, we show that $i_*$ is surjective.
    Let $\alpha$ be an arbitrary oriented loop in $E$.
    Let $l$ be an oriented longitude of $N(K)$.
    Suppose that $h: E \to V$ is the inclusion map and $h_*: H_1(E; \mathbb{Z}) \to H_1(V; \mathbb{Z})$ is the homomorphism induced by $h$.
    Since $|r(K)| = 1$, there is an integer $k \in \mathbb{Z}$ such that $h_*([\alpha] + k \cdot i_*([l])) = 0 \in H_1(V; \mathbb{Z})$. 
    
    Suppose that $F$ is an embedded annulus in $E$ such that $\beta_1 = \partial F \cap \alpha$ is a sub-arc of $\alpha$ and $m = F \cap \partial N(K)$ is the meridian of $N(K)$ (See Figure \ref{fig:hom}).
    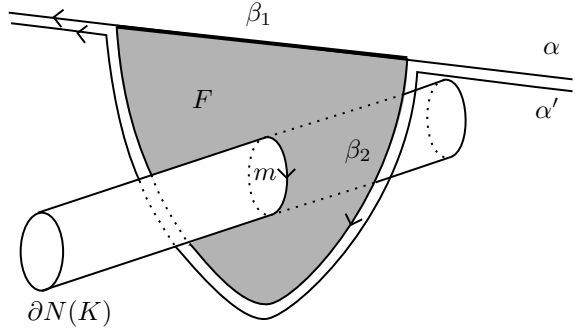
\begin{figure}[htbp]
        \centering
        \tikzset{every picture/.style={line width=0.75pt}} 

\begin{tikzpicture}[x=0.75pt,y=0.75pt,yscale=-1,xscale=1]

\draw   (22.71,115) .. controls (28.63,115) and (33.43,123.81) .. (33.43,134.67) .. controls (33.43,145.53) and (28.63,154.34) .. (22.71,154.34) .. controls (16.8,154.34) and (12,145.53) .. (12,134.67) .. controls (12,123.81) and (16.8,115) .. (22.71,115) -- cycle ;
\draw    (20.71,115) -- (136,76.84) ;
\draw    (22.71,154.34) -- (138,116.18) ;
\draw    (226,47.84) .. controls (238.49,47.37) and (240.49,84.37) .. (228,87.18) ;
\draw  [dash pattern={on 0.84pt off 2.51pt}]  (226,47.84) .. controls (213.2,49.54) and (214.49,88.37) .. (228,87.18) ;
\draw    (6,14) -- (290.33,47.67) ;
\draw    (136,76.84) .. controls (148.49,76.37) and (150.49,113.37) .. (138,116.18) ;
\draw  [dash pattern={on 0.84pt off 2.51pt}]  (136,76.84) .. controls (123.2,78.54) and (124.49,117.37) .. (138,116.18) ;
\draw  [draw opacity=0][fill={rgb, 255:red, 0; green, 0; blue, 0 }  ,fill opacity=0.3 ] (60.43,21.05) .. controls (86.2,23.74) and (169.2,31.54) .. (207,37) .. controls (209.43,86.05) and (159.31,157.16) .. (137.43,161.05) .. controls (127.88,162.75) and (112.2,150.78) .. (97.46,130.43) .. controls (117.09,122.64) and (127.05,119.65) .. (138,116.18) .. controls (151,109.67) and (149,79) .. (136,76.84) .. controls (119,82.67) and (93,91.67) .. (78,95.94) .. controls (79.79,100.13) and (81.39,103.62) .. (82.87,106.65) .. controls (70.33,82.5) and (60.84,52.47) .. (60.43,21.05) -- cycle ;
\draw    (97.2,130.54) .. controls (101.7,135.04) and (122.2,163.54) .. (137.43,161.05) .. controls (152.66,158.57) and (203.2,105.54) .. (207,37) ;
\draw    (60.43,21.05) .. controls (62.2,47.54) and (67.2,72.54) .. (78,95.94) ;
\draw  [dash pattern={on 0.84pt off 2.51pt}]  (78,95.94) .. controls (82.2,108.54) and (90.2,120.54) .. (97.2,131.54) ;
\draw   (32.49,21.26) -- (28.69,16.3) -- (33.65,12.51) ;
\draw   (150.48,92.9) -- (146.05,97.3) -- (141.65,92.87) ;
\draw   (184.25,118.32) -- (178.64,121.07) -- (175.89,115.46) ;
\draw  [dash pattern={on 0.84pt off 2.51pt}]  (136,76.84) -- (205,55.33) ;
\draw    (205,55.33) -- (226,47.84) ;
\draw  [dash pattern={on 0.84pt off 2.51pt}]  (138,116.18) -- (191,99.67) ;
\draw    (191,99.67) -- (228,87.18) ;
\draw [line width=1.5]    (60.43,21.05) -- (207,37) ;
\draw    (90,132) .. controls (94.5,136.5) and (121,171) .. (136.23,168.52) .. controls (151.46,166.03) and (208.2,112.54) .. (212,44) .. controls (254.43,49.34) and (261.2,50.34) .. (290.33,53.67) ;
\draw    (7.2,19.34) .. controls (24.2,21.34) and (40.2,23.34) .. (55.2,25.34) .. controls (57.31,48.12) and (62.42,78.18) .. (72,98.94) ;
\draw  [dash pattern={on 0.84pt off 2.51pt}]  (72,98.94) .. controls (76.2,111.54) and (83,121) .. (90,132) ;
\draw   (44.49,29.26) -- (40.69,24.3) -- (45.65,20.51) ;

\draw (274,27.4) node [anchor=north west][inner sep=0.75pt]    {$\alpha $};
\draw (128,91.4) node [anchor=north west][inner sep=0.75pt]    {$m$};
\draw (124,7.4) node [anchor=north west][inner sep=0.75pt]    {$\beta _{1}$};
\draw (174,76.4) node [anchor=north west][inner sep=0.75pt]    {$\beta _{2}$};
\draw (97,51.4) node [anchor=north west][inner sep=0.75pt]    {$F$};
\draw (270,55.4) node [anchor=north west][inner sep=0.75pt]    {$\alpha '$};
\draw (14,157.4) node [anchor=north west][inner sep=0.75pt]    {$\partial N( K)$};

\end{tikzpicture}
        \caption{The arc $\alpha'$}
        \label{fig:hom}
    \end{figure}
    Let $\beta_2$ denote the arc $\partial F - (\text{int}\beta_1 \cup m)$.
    Choose an orientation of $F$ so that the orientation of $\beta_1$ induced by the orientation of $F$ is reverse to the orientation of $\alpha$.
    Assume that the orientations of $\beta_1$, $\beta_2$, and $m$ are induced by the orientation of $F$.
    Let $\alpha'$ denote the loop $(\alpha - \beta_1) \cup \beta_2$.
    Now, we have $[\alpha'] = [\alpha] + [m]$ in $H_1(E; \mathbb{Z})$, and $\alpha'$ is obtained by moving $\alpha$ across $N(K)$ once.
    Thus, there is an integer $k' \in \mathbb{Z}$ such that $[\alpha] + k \cdot i_*([l]) + k' \cdot i_*([m]) = 0 \in H_1(E; \mathbb{Z})$.
    Therefore, we have $[\alpha] = i_*(-k[l] - k'[m]) \in H_1(E; \mathbb{Z})$.
    This implies that $i_*$ is surjective.

    In a similar way, we can show that $j_*$ is bijective.
\end{proof}

\begin{lem}\label{lem:planarSurface}
    Let $E$ denote the exterior of a knot $K$ in the solid torus $V$.
    Assume that there is a properly embedded annulus $A$ in $E$ such that $\partial A$ meets both $\partial V$ and $\partial N(K)$, and $\partial A \cap \partial N(K)$ is the meridian of $N(K)$.
    Then $K$ is the core of $V$ if and only if there is a properly embedded planar surface $F$ in $E$ such that 
    $|\partial F \cap \partial N(K)| = 1$, and $\partial F \cap \partial N(K)$ is an essential loop in $\partial N(K)$ and not the meridian of $N(K)$.
\end{lem}
\begin{proof}
    Suppose that $K$ is the core of $V$.
    Since $K$ is parallel to $\partial V$, there is a properly embedded annulus $F$ such that 
    $|\partial F \cap \partial N(K)| = 1$, and $\partial F \cap \partial N(K)$ is an essential loop in $\partial N(K)$ and not the meridian of $N(K)$.

    Conversely, suppose that there is a properly embedded planar surface $F$ in $E$ such that 
    $|\partial F \cap \partial N(K)| = 1$, and $\partial F \cap \partial N(K)$ is an essential loop in $\partial N(K)$ and not the meridian of $N(K)$.
    Let $a$ denote the loop $\partial F \cap \partial N(K)$.
    Suppose that $|\partial F \cap \partial V| = n$, and let $b_1, \ldots, b_n$ be the components of $\partial F \cap \partial V$.
    We prove that $K$ is the core of $V$ by induction on $n$.
    If $n=1$, i.e. $F$ is the annulus, then we see that $K$ is the core of $V$ by Lemma \ref{lem:condition2}.

    Assume that $n \geq 2$.
    First, we consider the case where $F$ is compressible, i.e. there is a compression disk $\delta$ for $F$.
    The boundary $\partial \delta$ divides $F$ into the two sub-surfaces $d_1$ and $d_2$.
    Suppose that $d_1$ contains $\partial F \cap \partial N(K)$.
    Since $\partial \delta$ is essential in $F$, we have $|\partial (\delta \cup d_1) \cap \partial V| < |\partial F \cap \partial V|$.
    Using the induction hypothesis, $K$ is the core of $V$.

    Next, suppose that $F$ is incompressible.
    We take $F$ so that $|F \cap A|$ is minimal up to isotopy of $F$.
    Using Lemma \ref{lem:noBigon}, all bigons in $\partial E$ bounded by $\partial F$ and $\partial A$ are removed while keeping the minimality of $|F \cap A|$.
    
    We show that there is a $\partial$-compression disk for $F$ whose intersection with $F$ connects distinct boundary components $b_i$ and $b_j$ of $F$.
    Fix an orientation of $F$.
    The orientations of $a, b_1, \ldots, b_n$ are induced by the orientation of $F$.
    Now, we see that $[\partial F] = [a] + \sum_{i=1}^n[b_i] = 0 \in H_1(E; \mathbb{Z})$.
    Thus, we have $-[a] = \sum_{i=1}^{n} [b_i]$.
    Let $i:\partial N(K) \to E$ and $j:\partial V \to E$ denote the inclusion maps, 
    and let $i_*:H_1(\partial N(K); \mathbb{Z}) \to H_1(E; \mathbb{Z})$ and $j_*:H_1(\partial V; \mathbb{Z}) \to H_1(E; \mathbb{Z})$ denote the homomorphisms induced by $i$ and $j$, respectively.
    By Lemma \ref{lem:bijection}, $i_*$ and $j_*$ are bijective.
    Fix an orientation of $A$.
    Let $f:H_1(\partial N(K); \mathbb{Z}) \to \mathbb{Z} \oplus \mathbb{Z}$ and $g:H_1(\partial V; \mathbb{Z}) \to \mathbb{Z} \oplus \mathbb{Z}$ be isomorphisms 
    such that $f(\partial A \cap \partial N(K)) = (0,1)$, $g(\partial A \cap \partial V) = (0,1)$, and if $g(\alpha) = (1,0)$, then $f \circ i_*^{-1} \circ j_* (\alpha) = (1,0)$, where $\alpha$ is a loop in $\partial V$.
    We see that $\partial A \cap \partial V$ and $\partial A \cap \partial N(K)$ are homologous in $E$.
    This implies that $f \circ i_*^{-1} \circ j_* (\partial A \cap \partial V) = (0,1)$, and so $f \circ i_*^{-1} \circ j_* \circ g^{-1}$ is the identity map.
    Suppose that $(x, y) = f \circ i_*^{-1}(-[a])$ and for each $b_i$, $(p_i, q_i) = g \circ j_*^{-1} ([b_i])$.
    Since $F$ is incompressible, the loops $b_1, \ldots, b_n$ are essential in $\partial V$.
    Thus, there is a pair of integers $(p, q) \in \mathbb{Z} \oplus \mathbb{Z}$ such that $(p, q)$ is equal to $(p_i, q_i)$ or $(-p_i, -q_i)$ for each $i$.
    This implies that there is an integer $k \in \mathbb{Z}$ such that $\sum_{i=1}^{n} (p_i, q_i) = (kp, kq)$.
    Now, we have
    $    (x, y) = f \circ i_*^{-1}(-[a]) 
               = \sum_{i=1}^n f \circ i_*^{-1}([b_i])
               = \sum_{i=1}^n f \circ i_*^{-1} \circ j_* \circ g^{-1}(p_i, q_i)
               = \sum_{i=1}^n (p_i, q_i)
               = (kp, kq).
    $
    Since $a$ is an essential loop in $\partial N(K)$ and not the meridian of $N(K)$, $x$ and $y$ satisfy either of the following:
    \begin{itemize}
        \item $(x, y) = (1,0)$ or 
        \item $x \neq 0$, $y \neq 0$, and $x$ and $y$ are relatively prime.
    \end{itemize}
    Thus, it follows that $k = 1$, and so $(x, y) = (p, q)$.
    This implies that $|\partial F \cap \partial A \cap \partial N(K)| = p$ and $|\partial F \cap \partial A \cap \partial V| = np$.
    By the assumption that $n \geq 2$, we have $|\partial F \cap \partial A \cap \partial N(K)| < |\partial F \cap \partial A \cap \partial V|$.
    This implies that there is an arc $\alpha$ of $F \cap A$ whose endpoints are in $\partial A \cap \partial V$.
    Let $\delta$ denote the disk in $A$ whose boundary consists of $\alpha$ and a sub-arc of $\partial A$.

    We show that there are no loops of $F \cap A$ in $\delta$.
    Suppose that there are loops of $F \cap A$ in $\delta$.
    Let $\delta'$ be an innermost disk in $\delta$ and $\alpha'$ denote the boundary of $\delta'$.
    Since $F$ is incompressible, $\alpha'$ bounds a disk $\delta''$ in $F$.
    From the irreducibility of $E$, the properly embedded $2$-sphere $\delta' \cup \delta''$  bounds a $3$-ball in $E$.
    Thus, we can remove $\alpha'$ from $F\cap A$ by an isotopy of $F$ along the $3$-ball.
    This contradicts the minimality of $|F \cap A|$.
    Therefore, there are no loops of $F \cap A$ in $\delta$.

    Let $\alpha''$ denote an outermost arc in $\delta$ and $\delta''$ denote the disk bounded by $\alpha''$ and a sub-arc of $\partial A$.
    Since there are no bigons in $\partial E$ bounded by $\partial F$ and $\partial A$, $\alpha''$ connects distinct boundary components of $F$.
    Since $\delta'' \cap F = \emptyset$, the disk $\delta''$ is a $\partial$-compression disk for $F$ such that $F \cap \delta''$ connects distinct boundary components of $F$.

    Let $F'$ denote the surface obtained from $F$ by $\partial$-compression using the boundary compression disk $\delta''$.
    Then we have $|\partial F \cap \partial V| > |\partial F' \cap \partial V|$.
    Using the induction hypothesis, $K$ is the core of $V$.
\end{proof}

\begin{lem}\label{lem:condition1}
    Let $V$ and $W$ be two solid tori with the meridians $m_V$ and $m_W$, respectively.
    Let $K$ be a knot in $V$ and $E_V$ denote the exterior $V - \text{int}N(K)$.
    Assume that $A$ is a properly embedded annulus in $E_V$ such that $\partial A$ meets both $\partial N(K)$ and $\partial V$, and $\partial A \cap \partial N(K)$ is the meridian of $N(K)$.
    Let $\alpha$ be an essential simple closed curve in $\partial V$ such that $\alpha$ is not the meridian of $V$.
    Suppose that $M$ is the $3$-manifold obtained by gluing $\partial V$ and $\partial W$ so that $\alpha$ and $m_W$ are identified, 
    $K_M$ is the knot in $M$ obtained from $K$, and $E_M$ is the exterior $M - \text{int}N(K_M)$.
    Then $K$ is the core of $V$ if and only if there is a properly embedded essential disk $D$ in $E_M$.
\end{lem}
\begin{proof}
    Suppose that $K$ is the core of $V$.
    We see that $E_M$ is homeomorphic to the solid torus.
    Thus, there is a properly embedded essential disk $D$ in $E_M$.

    Conversely, suppose that there is a properly embedded essential disk $D$ in $E_M$.
    We denote the properly embedded torus $\partial V = \partial W$ in $E_M$ by $T$.
    Assume that $D$ and $T$ intersect transversely.
    
    We show that $\partial D$ is not the meridian of $N(K)$.
    Let $\delta$ denote the meridian disk of $N(K)$ such that $\partial \delta = \partial D$.
    Then we see that $D \cup \delta$ is a non-separating $2$-sphere in $M$.
    On the other hand, $M$ is irreducible since $\alpha$ is not the meridian of $V$.
    This is a contradiction.
    Thus, $\partial D$ is not the meridian of $N(K)$.

    Let $F$ denote the component of $V \cap D$ such that $\partial F$ contains $\partial D$.
    Then $F$ is a properly embedded planar surface in $E_V$, $|\partial F \cap \partial N(K)| = 1$, and $\partial F \cap \partial N(K)$ is not the meridian of $N(K)$.
    Using Lemma \ref{lem:planarSurface}, $K$ is the core of $V$. 
\end{proof}

Now, we are ready to show Lemma \ref{lem:condOfCore}.
\begin{proof}[Proof of Lemma \ref{lem:condOfCore}]
    Suppose that $K$ is the core of $V$.
    Then we have a properly embedded essential annulus $A$ in $E_V$ such that $\partial A$ meets both $\partial N(K)$ and $\partial V$, and $\partial A \cap \partial N(K)$ is the meridian of $N(K)$.
    Using Lemma \ref{lem:condition1}, there is a properly embedded essential disk $D$ in $E_M$.

    Conversely, suppose that there is a properly embedded essential annulus in $E_V$ such that $\partial A$ meets both $\partial V$ and $\partial N(K)$, and 
    there is a properly embedded essential disk $D$ in $E_M$.
    If $\partial A \cap \partial N(K)$ is not the meridian of $N(K)$, then we see that $K$ is the core of $V$ by Lemma \ref{lem:condition2}.
    If $\partial A \cap \partial N(K)$ is the meridian of $N(K)$, then $K$ is the core of $V$ from Lemma \ref{lem:condition1}.
\end{proof}
\section{An algorithm for the solid torus core recognition problem}
In this section, we describe an algorithm for the solid torus core recognition problem.

\subsection{Algorithms for deciding whether disks and annuli in the exterior of a knot are essential}
Let $\mathcal{T}_{E_M}$ be an $n$-tetrahedra triangulation of the exterior $E_M$ of a knot $K$ in a compact $3$-manifold $M$.
Suppose that $E_M$ is irreducible.
First, we describe that if $D$ is a normal surface with respect to $\mathcal{T}_{E_M}$, then it can be verified that $D$ is an essential disk in $E_M$ in polynomial time of $n$.

\begin{lem}\label{lem:H1gen}
    Let $\mathcal{T}_T$ be a triangulation of the torus $T$.
    Suppose that $\mathcal{T}_T$ contains $n$ triangles.
    Then there is a polynomial time algorithm that, given $\mathcal{T}_T$, outputs simple closed curves $m$ and $l$ in the $1$-skeleton $\mathcal{T}_T^{(1)}$ such that the homology classes $[m]$ and $[l]$ generate $H_1(T; \mathbb{Z})$. 
\end{lem}
\begin{proof}
    Let $n_e$ and $n_v$ denote the number of edges and vertices of $\mathcal{T}_T$, respectively.
    There is a $\mathcal{O}(n_e+n_v) = \mathcal{O}(n)$ time algorithm that, given $\mathcal{T}_T^{(1)}$ and a vertex of $\mathcal{T}_T$, outputs a non-contractive simple closed curve passing the given vertex if it exists (\cite{CVL}).
    Thus, we can obtain an essential simple closed curve $m$ in $\mathcal{T}_T^{(1)}$ in polynomial time of $n$.
    Suppose that $v$ is a vertex contained in $m$.
    Let $\mathcal{T}_T'$ denote the triangulation of the annulus obtained by cutting $\mathcal{T}_T$ along $m$, and let $v^+$ and $v^-$ denote the vertices of $\mathcal{T}_T'$ obtained from $v$.
    A simple path $p$ in $\mathcal{T}_T'^{(1)}$ connecting $v^+$ and $v^-$ is obtained in $\mathcal{O}(n)$ time.
    Let $l$ denote the simple closed curve in $\mathcal{T}_T^{(1)}$ obtained from $p$ by gluing $v^+$ and $v^-$.
    Now, we see that $[m]$ and $[l]$ generate $H_1(T; \mathbb{Z})$.
    This completes the proof.
\end{proof}

Let $\mathcal{T}_S$ be a triangulation of a compact surface $S$.
A properly embedded simple curve $\alpha$ in $S$ is called a \textit{normal curve} with respect to $\mathcal{T}_S$
if for each triangle $t_i$ of $\mathcal{T}_S$, $\alpha \cap t_i$ is a collection of elementary arcs of $t_i$.
In a similar way of normal surfaces, a normal curve with respect to $\mathcal{T}_S$ is represented by a vector $\bm{x} = (x_1, \ldots, x_{3n}) \in \mathbb{Z}^{3n}$, where $n$ is the number of triangles in $\mathcal{T}_S$.

\begin{lem}\label{lem:isEssentialLoop}
    Let $\mathcal{T}_T$ be a triangulation of the torus $T$.
    Suppose that $\mathcal{T}_T$ contains $n$ triangles.
    Let $\alpha$ be a normal curve with respect to $\mathcal{T}_T$ that is represented by a vector $\bm{x} = (x_1, \ldots, x_{3n}) \in \mathbb{Z}^{3n}$.
    Assume that each $x_i$ is at most $2^{\mathcal{O}(n)}$.
    Then there is an algorithm that, given $\mathcal{T}_T$ and $\bm{x}$, decides whether $\alpha$ is essential in $T$ in polynomial time of $n$.
\end{lem}
\begin{proof}
    By Lemma \ref{lem:H1gen}, simple closed curves $m$ and $l$ in $\mathcal{T}_T^{(1)}$ such that the homology classes $[m]$ and $[l]$ generate $H_1(T; \mathbb{Z})$ are obtained in polynomial time of $n$.
    Since $\alpha$ is a normal curve with respect to $\mathcal{T}_T$, $\alpha$ and $m$ intersect transversely.
    For a similar reason, $\alpha$ and $l$ also intersect transversely.
    The simple closed curve $\alpha$ is essential in $T$ if and only if either of $|m \cap \alpha|$ or $|l \cap \alpha|$ is odd.
    Since each $x_i$ is at most $2^{\mathcal{O}(n)}$, $|m \cap \alpha|$ and $|l \cap \alpha|$ are calculated in polynomial time of $n$.
    Thus, we can decide whether $\alpha$ is essential in $T$ in polynomial time of $n$.
\end{proof}

\begin{lem}\label{lem:algEssentialDisk}
    Let $\mathcal{T}_{E_M}$ be an $n$-tetrahedra triangulation of the exterior $E_M$ of a knot $K$ in a compact $3$-manifold $M$.
    Suppose that $E_M$ is irreducible. 
    Let $D$ be a normal surface with respect to $\mathcal{T}_{E_M}$ represented by a vector $\bm{x} = (x_1, \ldots, x_{7n}) \in \mathbb{Z}^{7n}$.
    Suppose that each coordinate $x_i$ is less than or equal to $2^{7n-1}$. 
    Then there is an algorithm that, given $\mathcal{T}_{E_M}$ and $\bm{x}$, decides whether $D$ is an essential disk in $E_M$ in polynomial time of $n$.
\end{lem}
\begin{proof}
    The normal surface $D$ is the disk if and only if $D$ is connected, $\chi(D) = 1$, and $\partial D \neq \emptyset$.
    There is an algorithm that, given $\mathcal{T}_{E_M}$ and $\bm{x}$, outputs the number of components of $D$ in polynomial time of $n \log 2^{7n-1}$ = $(7n^2-n) \log 2$ (\cite{AHT}).
    Thus, we can verify whether $D$ is connected in polynomial time of $n$.
    By Lemma \ref{lem:euler}, $\chi(D)$ is calculated in polynomial time of $n$, and it can be verified that $\partial D \neq \emptyset$ in polynomial time of $n$.
    Therefore, we can check whether $D$ is the disk in polynomial time of $n$.

    The disk $D$ is essential in $E_M$ if and only if $\partial D$ is essential in $\partial E_M$ since $E_M$ is irreducible.
    Let $\partial \mathcal{T}_{E_M}$ denote the triangulation of $\partial E_M$ obtained from $\mathcal{T}_{E_M}$ and $n'$ denote the number of triangles in $\partial \mathcal{T}_{E_M}$.
    Suppose that $\bm{y} = (y_1, \ldots, y_{3n'}) \in \mathbb{Z}^{3n'}$ is the representation vector of the normal curve $\partial D$ with respect to $\partial \mathcal{T}_{E_M}$.
    Since $x_i \leq 2^{7n-1}$ for each $x_i$, we see that $y_i$ is at most $2^{\mathcal{O}(n)}$.
    Thus, we can verify that $\partial D$ is essential in $\partial E_M$ in polynomial time of $n$ by Lemma \ref{lem:isEssentialLoop}.
\end{proof}

Let $\mathcal{T}_E$ be an $n$-tetrahedra triangulation of the exterior $E$ of a knot $K$ in the solid torus $V$ and 
$A$ be a normal surface with respect to $\mathcal{T}_E$.
Next, we describe that it can be verified that $A$ is an essential annulus in $E$ such that $\partial A$ meets both $\partial N(K)$ and $\partial V$ in polynomial time of $n$.
\begin{lem}\label{lem:essentialAnnulus}
    Let $E$ denote the exterior of a knot $K$ in the solid torus $V$ and $A$ be a properly embedded annulus in $E$ such that $\partial A$ meets both $\partial N(K)$ and $\partial V$.
    Assume that $|r(K)| \neq 0$.
    Then $A$ is essential if and only if $\partial A \cap \partial V$ is essential in $\partial V$.
\end{lem}
\begin{proof}
    Suppose that $\partial A \cap \partial V$ is inessential in $\partial V$.
    Then there is a disk $D$ in $\partial V$ bounded by $\partial A$, and the disk obtained by pushing $D$ to the interior of $E$ is a compression disk for $A$.
    Thus, $A$ is inessential.

    Conversely, suppose that $A$ is inessential.
    Since $\partial A$ meets both $\partial N(K)$ and $\partial V$, $A$ is $\partial$-incompressible and not parallel to $\partial E$.
    This implies that $A$ is compressible, i.e. there is a compression disk $D$ for $A$.
    Let $A'$ and $A''$ denote the disks obtained by compressing $A$ using $D$.
    Suppose that $\partial A'$ meets $\partial V$.
    By the assumption that $|r(K)| \neq 0$, $E$ is $\partial$-irreducible by Lemma \ref{lem:irr}.
    Therefore, $\partial A' = \partial A \cap \partial V$ is inessential in $\partial V$.
\end{proof}

\begin{lem}\label{lem:algEssentialAnnulus}
    Let $\mathcal{T}_E$ be an $n$-tetrahedra triangulation of the exterior $E$ of a knot $K$ in the solid torus $V$ and 
    $A$ be a normal surface with respect to $\mathcal{T}_E$ represented by a vector $\bm{x} = (x_1, \ldots, x_{7n}) \in \mathbb{Z}^{7n}$.
    Suppose that each coordinate $x_i$ is less than or equal to $2^{7n-1}$.
    Then there is a polynomial time algorithm that, given $\mathcal{T}_E$ and $\bm{x}$, 
    decides whether $A$ is an essential annulus in $E$ such that $\partial A$ meets both $\partial N(K)$ and $\partial V$.
\end{lem}
\begin{proof}
    The normal surface $A$ is a properly embedded annulus in $E$ such that $\partial A$ meets both $\partial N(K)$ and $\partial V$ 
    if and only if $A$ is connected, $\chi(A) = 0$, and $\partial A$ meets both $\partial N(K)$ and $\partial V$.
    The number of components of $A$ is calculated in polynomial time of $n$ by \cite{AHT}.
    Thus, we can verify whether $A$ is connected in polynomial time of $n$.
    The Euler characteristic $\chi(A)$ can be calculated in polynomial time of $n$ by Lemma \ref{lem:euler}.
    It also can be verified whether $\partial A$ meets both $\partial N(K)$ and $\partial V$ in polynomial time of $n$.
    Therefore, it can be verified that $A$ is a properly embedded annulus in $E$ such that $\partial A$ meets both $\partial N(K)$ and $\partial V$ in polynomial time of $n$.

    By Lemma \ref{lem:essentialAnnulus}, the annulus $A$ is essential in $E$ if and only if $\partial F \cap \partial V$ is essential in $\partial V$.
    By Lemma \ref{lem:isEssentialLoop}, we can decide whether $\partial A \cap \partial V$ is essential in $\partial V$ in polynomial time of $n$.
    Thus, it can be verified whether $A$ is essential in $E$ in polynomial time of $n$.
\end{proof}

\subsection{An algorithm for gluing the boundaries of two solid tori}
In order to solve the solid torus core recognition problem using Lemma \ref{lem:condOfCore}, 
we describe an algorithm for gluing the boundaries of two solid tori in polynomial time.
\begin{lem}\label{lem:glue2handle}
    Let $\mathcal{T}_M$ be an $n$-tetrahedra triangulation of a compact $3$-manifold $M$ with non-empty boundary and 
    $\alpha$ be a simple closed curve in $\partial M$ represented by a collection of edges of $\mathcal{T}_M^{(1)}$.
    Let $N$ be the $3$-manifold obtained by gluing a $2$-handle along $\alpha$.
    Then there is an algorithm that, given $\mathcal{T}_M$ and $\alpha$, outputs a triangulation $\mathcal{T}_{N}$ of $N$ in polynomial time of $n$.
    Moreover, $\text{size}(\mathcal{T}_{N})$ is at most $\mathcal{O}(n)$.
\end{lem}
\begin{proof}
    We can obtain $\mathcal{T}_{N}$ as follows.
    Let $\mathcal{T}''_M$ denote the triangulation obtained by barycentrically subdividing $\mathcal{T}_M$ twice.
    The barycentric subdivision is performed in $\mathcal{O}(n)$ time, and $\text{size}(\mathcal{T}''_M)$ is at most $\mathcal{O}(n)$.
    Let $\partial \mathcal{T}''_{M}$ denote the triangulation of $\partial M$ obtained from $\mathcal{T}_M''$.
    Let $\mathcal{A}$ denote the triangulation of $A$ consists of the faces of $\partial \mathcal{T}''_{M}$ that meet $\alpha$, 
    where $A$ is the annulus in $\partial M$ such that $\alpha$ is the core of $A$.
    We can obtain $\mathcal{A}$ in $\mathcal{O}(\text{size}(\partial \mathcal{T}_M'')) = \mathcal{O}(n)$ time, and $\text{size}(\mathcal{A})$ is at most $\mathcal{O}(n)$.
    Let $B_1, \ldots, B_{\text{size}({\mathcal{A}})}$ denote the triangulated $3$-balls depicted as in Figure \ref{fig:triangulated3-ball}.
    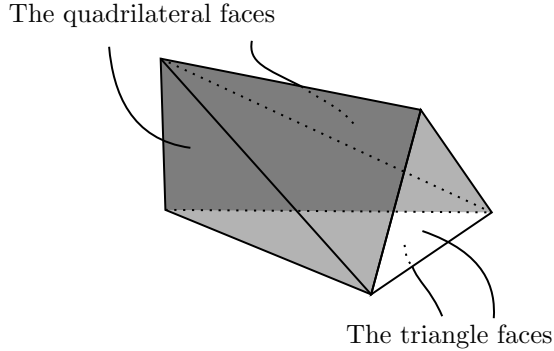
\begin{figure}[htbp]
        \centering
        \tikzset{every picture/.style={line width=0.75pt}} 

\begin{tikzpicture}[x=0.75pt,y=0.75pt,yscale=-1,xscale=1]

\draw   (228.21,68.21) -- (264,119.95) -- (203.16,161.34) -- cycle ;
\draw  [line width=0.75]  (99.39,118.66) -- (97,42.34) -- (228.21,68.21) -- (203.16,161.34) -- cycle ;
\draw    (97,42.34) -- (203.16,161.34) ;
\draw [line width=0.75]  [dash pattern={on 0.84pt off 2.51pt}]  (99.39,118.66) -- (264,119.95) ;
\draw  [dash pattern={on 0.84pt off 2.51pt}]  (97,42.34) -- (264,119.95) ;
\draw  [draw opacity=0][fill={rgb, 255:red, 0; green, 0; blue, 0 }  ,fill opacity=0.3 ] (99.39,118.66) -- (97,42.34) -- (228.21,68.21) -- (203.16,161.34) -- cycle ;
\draw  [draw opacity=0][fill={rgb, 255:red, 0; green, 0; blue, 0 }  ,fill opacity=0.3 ] (99.39,118.66) -- (97,42.34) -- (228.21,68.21) -- (264,119.95) -- cycle ;
\draw    (71,36.34) .. controls (73,59.34) and (82,83.34) .. (112,87.34) ;
\draw    (143,33.34) .. controls (141,44.34) and (150,47.34) .. (168,56.34) ;
\draw  [dash pattern={on 0.84pt off 2.51pt}]  (168,56.34) .. controls (179,64.34) and (188,66.34) .. (196,77.34) ;
\draw    (229,129.34) .. controls (261,140.34) and (266,158.34) .. (265,173.34) ;
\draw    (224,147.34) .. controls (227,155.34) and (233,158.34) .. (239,172.34) ;
\draw  [dash pattern={on 0.84pt off 2.51pt}]  (220,136.34) .. controls (221,143.34) and (221,143.34) .. (224,147.34) ;

\draw (19,13) node [anchor=north west][inner sep=0.75pt]   [align=left] {The quadrilateral faces};
\draw (189,175) node [anchor=north west][inner sep=0.75pt]   [align=left] {The triangle faces};

\end{tikzpicture}
        \caption{A triangulated $3$-ball $B_i$}
        \label{fig:triangulated3-ball}
    \end{figure}
    Each $B_i$ has the two triangle faces and the two triangulated quadrilateral faces.
    Then $\mathcal{T}_N$ is obtained by gluing a triangle face of $B_i$ and a face of $\mathcal{A}$, and gluing the faces of adjacent triangulated $3$-balls $B_i$ and $B_j$.
    If adjacent quadrilateral faces cannot be glued, then the faces are glued by adding a tetrahedron (See Figure \ref{fig:layerd}).
    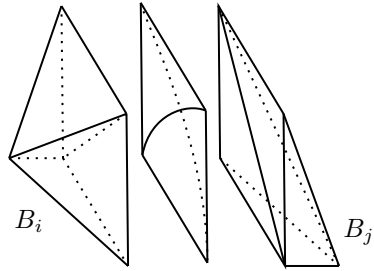
\begin{figure}[htbp]
        \centering
        \tikzset{every picture/.style={line width=0.75pt}} 
\hspace{10mm}
\begin{tikzpicture}[x=0.75pt,y=0.75pt,yscale=-0.8,xscale=0.8]

\draw   (137,8.34) -- (138,104.34) -- (179,172.34) -- (178,76.34) -- cycle ;
\draw    (88,6.34) -- (89,102.34) ;
\draw    (129,74.34) -- (130,170.34) ;
\draw    (88,6.34) -- (129,74.34) ;
\draw    (89,102.34) -- (130,170.34) ;
\draw    (89,102.34) .. controls (94,76.34) and (117,69.34) .. (129,74.34) ;
\draw  [dash pattern={on 0.84pt off 2.51pt}]  (88,6.34) .. controls (89,15.34) and (130,92.34) .. (130,170.34) ;
\draw    (137,8.34) -- (179,172.34) ;
\draw    (179,172.34) -- (213,172.34) ;
\draw    (178,76.34) -- (213,172.34) ;
\draw  [dash pattern={on 0.84pt off 2.51pt}]  (137,8.34) -- (213,172.34) ;
\draw  [dash pattern={on 0.84pt off 2.51pt}]  (138,104.34) -- (213,172.34) ;
\draw    (80,172.34) -- (5,104.34) ;
\draw  [dash pattern={on 0.84pt off 2.51pt}]  (39,104.34) -- (80,172.34) ;
\draw  [dash pattern={on 0.84pt off 2.51pt}]  (39,104.34) -- (38,8.34) ;
\draw    (80,172.34) -- (79,76.34) ;
\draw    (38,8.34) -- (79,76.34) ;
\draw  [dash pattern={on 0.84pt off 2.51pt}]  (79,76.34) -- (39,104.34) ;
\draw  [dash pattern={on 0.84pt off 2.51pt}]  (5,104.34) -- (39,104.34) ;
\draw    (5,104.34) -- (38,8.34) ;
\draw    (79,76.34) -- (5,104.34) ;

\draw (17.59,144) node    {$B_{i}$};
\draw (225.59,151) node    {$B_{j}$};

\end{tikzpicture}
        \caption{Gluing adjacent quadrilateral faces}
        \label{fig:layerd}
    \end{figure}
    This procedure is performed in $\mathcal{O}(n)$ time and increase the number of tetrahedra by at most $\mathcal{O}(n)$.
    Thus, $\mathcal{T}_N$ is obtained in polynomial time of $n$, and $\text{size}(\mathcal{T}_N)$ is at most $\mathcal{O}(n)$. 
\end{proof}

\begin{lem}\label{lem:algGluingSolidTori}
    Let $\mathcal{T}_V$ be an $n$-tetrahedra triangulation of a solid torus $V$.
    Suppose that $K$ is a knot in $V$ represented by a collection of edges of $\mathcal{T}_V^{(1)}$ and $\alpha$ is a simple closed curve in $\partial V$ represented by a collection of edges of $\mathcal{T}_V^{(1)}$.
    Let $W$ be a solid torus with the meridian $m_W$.
    Let $M$ denote the $3$-manifold obtained by gluing $\partial V$ and $\partial W$ so that $\alpha$ and $m_W$ are identified, and let $K_M$ denote the knot in $M$ obtained from $K$.
    Then there is an algorithm that, given $\mathcal{T}_V$, $K$, and $\alpha$, outputs a triangulation $\mathcal{T}_M$ of $M$ and the knot $K_M$ represented by a collection of edges of $\mathcal{T}_M^{(1)}$ in polynomial time of $n$.
    Moreover, $\text{size}({\mathcal{T}_M})$ is at most $\mathcal{O}(n)$.
\end{lem}
\begin{proof}
    We obtain $\mathcal{T}_M$ and $K_M$ as follows.
    Let $V'$ be the $3$-manifold obtained from $V$ by gluing a $2$-handle along $\alpha$.
    Using Lemma \ref{lem:glue2handle}, a triangulation $\mathcal{T}_{V'}$ of $V'$ is obtained from $\mathcal{T}_V$ in polynomial time of $n$, and $\text{size}(\mathcal{T}_{V'})$ is at most $\mathcal{O}(n)$.
    Then we obtain $\mathcal{T}_M$ by taking the cone of $\partial \mathcal{T}_{V'}$, where $\partial \mathcal{T}_{V'}$ is a triangulation of $\partial V'$ obtained from $\mathcal{T}_{V'}$.
    This procedure increase the number of tetrahedra by $\mathcal{O}(\text{size}(\partial \mathcal{T}_{V'})) = \mathcal{O}(n)$.
    Thus, $\mathcal{T}_M$ is obtained in polynomial time of $n$, $\mathcal{T}_M^{(1)}$ contains $K_M$, and $\text{size}(\mathcal{T}_M)$ is at most $\mathcal{O}(n)$.
\end{proof}

\subsection{Proof of Theorem \ref{thm:main}}
Now, we are ready to show Theorem \ref{thm:main}.

\begin{proof}[Proof of Theorem \ref{thm:main}]
    Let $\mathcal{T}_V$ be an $n$-tetrahedra triangulation of the solid torus $V$ and $K$ be a knot in $V$ represented by a collection of edges of $\mathcal{T}_V^{(1)}$.
    We consider the following non-deterministic algorithm.
    \begin{enumerate}
        \item Check whether $|r(K)| = 1$. If $|r(K)| \neq 1$, then output ``no''.
        \item Construct a triangulation $\mathcal{T}_{E_V}$ of the exterior $E_V$ of $K$, and let $n_1 = \text{size}(\mathcal{T}_{E_V})$.
        \item Guess a vector $\bm{x} = (x_1, \ldots, x_{7n_1}) \in \mathbb{Z}^{7n_1}$ such that each coordinate $x_i$ is less than or equal to $2^{7n_1-1}$.
        \item If $\bm{x}$ represents a normal surface with respect to $\mathcal{T}_{E_V}$, then let $A$ denote it. Otherwise output ``no''.
        \item If $A$ is not an essential annulus in $E_V$ such that $\partial A$ meets both $\partial N(K)$ and $\partial V$, then output ``no''.
        \item Take an essential simple closed curve $\alpha$ in $\partial V$ such that $\alpha$ is not the meridian of $V$ and $\alpha$ is contained in $\mathcal{T}_V^{(1)}$.
        \item Construct a triangulation $\mathcal{T}_M$ and the knot $K_M$, where $\mathcal{T}_M$ is a triangulation of the $3$-manifold $M$ obtained by gluing $\partial V$ and the boundary of a solid torus $W$ so that $\alpha$ is the meridian of $W$ and $K_M$ is the knot in $M$ obtained from $K$.
        \item Construct a triangulation $\mathcal{T}_{E_M}$ of the exterior $E_M = M - \text{int}N(K_M)$, and let $n_2 = \text{size}(\mathcal{T}_{E_M})$.
        \item Guess a vector $\bm{y} = (y_1, \ldots, y_{7n_2}) \in \mathbb{Z}^{7n_2}$ such that each coordinate $y_i$ is less than or equal to $2^{7n_2-1}$.
        \item If $\bm{y}$ represents a normal surface with respect to $\mathcal{T}_{E_M}$, then let $D$ denote it. Otherwise output ``no''.
        \item If $D$ is an essential disk in $E_M$, then output ``yes''. Otherwise output ``no''.
    \end{enumerate}

    First, we show that the above algorithm outputs ``yes'' if and only if $K$ is the core of $V$.
    Suppose that $K$ is the core of $V$.
    Since $|r(K)| = 1$, the $1$st step does not output ``no''.
    From Lemma \ref{lem:vertexEssentialAnnulus}, there is an essential annulus $A$ such that $A$ is a vertex surface with respect to $\mathcal{T}_E$.
    Let $\bm{x} = (x_1, \ldots, x_{7n_1}) \in \mathbb{Z}^{7n_1}$ denote the vector representation of $F$.
    Using Theorem \ref{thm:vertexBound}, we see that each coordinate $x_i$ is at most $2^{7n_1-1}$.
    This implies that we can guess the vector $\bm{x}$ representing the essential annulus $F$ in $E$ in the $3$rd step.
    Thus, the $4$th step and $5$th step do not output ``no''.
    From Theorem \ref{thm:compDisk} and Lemma \ref{lem:condOfCore}, there is an essential disk $D$ in $E_M$ such that $D$ is a vertex surface with respect to $\mathcal{T}_{E_M}$.
    Since $D$ is a vertex surface, we can guess the vector representation $\bm{y}$ of $D$ in the $9$th step, and the $10$th step does not output ``no''.
    Since $D$ is an essential disk, the $11$th step outputs ``yes''.
    
    Conversely, suppose that $K$ is not the core of $V$.
    If $|r(K)| \neq 1$, then the $1$st step outputs ``no''.
    Suppose that $|r(K)| = 1$.
    From Lemma \ref{lem:condOfCore}, there are no properly embedded essential annuli in $E_V$ such that the annuli meet both $\partial N(K)$ and $\partial V$ or 
    there are no properly embedded essential disks in $E_M$.
    Therefore, the algorithm outputs ``no'' in the $4$th, $5$th, $10$th, or $11$th step.

    Next, we show that the running time of the above algorithm is bounded by a polynomial of $n$.
    The $1$st step is performed in polynomial time of $n$ by Lemma \ref{lem:rotationNum}.
    Using Lemma \ref{lem:exterior}, the $2$nd step is performed in polynomial time of $n$, and $n_1$ is at most $\mathcal{O}(n)$.
    Since each coordinate $x_i$ is less than or equal to $2^{7n_1-1}$, $\bm{x}$ is represented by a binary code whose length is at most $\mathcal{O}(n_1^2)$.
    Thus, we can guess $\bm{x}$ in $\mathcal{O}(n_1^2) = \mathcal{O}(n^2)$ time in the $3$rd step.
    By Lemma \ref{lem:normal_alg}, the $4$th step runs in polynomial time of $n$.
    Lemma \ref{lem:algEssentialAnnulus} implies that the $5$th step runs in polynomial time of $n$.
    From Lemma \ref{lem:H1gen}, two simple closed curves $m$ and $l$ in $\partial V$ such that $[m]$ and $[l]$ generate $H_1(\partial V; \mathbb{Z})$ is obtained in polynomial time of $n$.
    Since at least one of $m$ and $l$ is not the meridian of $V$, 
    we can obtain a simple closed curve $\alpha$ in $\partial V$ such that $\alpha$ is not the meridian of $V$ by calculating the homology classes $[m]$ and $[l]$ in $H_1(V; \mathbb{Z})$.
    Therefore, the $6$th step runs in polynomial time of $n$.
    By Lemma \ref{lem:algGluingSolidTori}, the $7$th step is performed in polynomial time of $n$, and $\text{size}(\mathcal{T}_M)$ is at most $\mathcal{O}(n)$.
    A triangulation $\mathcal{T}_{E_M}$ of the exterior $E_M = M - \text{int}N(K_M)$ is obtained by barycentrically subdividing $\mathcal{T}_M$ twice and removing the tetrahedra containing $K_M$.
    Thus, we obtain $\mathcal{T}_{E_M}$ in polynomial time of $n$, and $\text{size}(\mathcal{T}_{E_M})$ is at most $\mathcal{O}(n)$.
    In a similar way of the $3$rd step and the $4$th step, the $9$th step and the $10$th step run in polynomial time of $n$.
    Since $|r(K)| \neq 0$, $E_M$ is irreducible.
    Thus, the $11$th step runs in polynomial time of $n$ by Lemma \ref{lem:algEssentialDisk}.
    Since each step is performed in polynomial time of $n$, the above algorithm runs in polynomial time of $n$.

    Now, we see that there is a non-deterministic polynomial time algorithm for the solid torus core recognition problem.
    Therefore, this problem is in \textbf{NP}.
\end{proof}

\subsection{Proof of Theorem \ref{thm:main2}}
For every compact surface $\Sigma$, the \textit{$\Sigma \times [0,1]$ recognition problem} is the problem of determining that 
the underlying $3$-manifold of a given triangulation is homeomorphic to $\Sigma \times [0,1]$.
Haraway and Hoffman showed that this problem is in $\textbf{co-NP}$ among orientable irreducible $3$-manifolds.
\begin{thm}[Haraway-Hoffman \cite{HH}]\label{thm:thickSurco_NP}
    For every compact surface $\Sigma$, the $\Sigma \times [0,1]$ recognition problem is in $\textbf{co-NP}$ among orientable irreducible $3$-manifolds.
\end{thm}

\begin{proof}[Proof of Theorem \ref{thm:main2}]
    Let $\mathcal{T}_V$ be an $n$-tetrahedra triangulation of the solid torus $V$ and $K$ be a knot in $V$ represented by a collection of edges of $\mathcal{T}_V^{(1)}$.
    We consider the following  non-deterministic algorithm.
    \begin{enumerate}
        \item If $|r(K)| \neq 1$, then output ``yes''.
        \item Construct a triangulation $\mathcal{T}_E$ of the exterior $E = V - \text{int}N(K)$.
        \item If $E$ is not homeomorphic to $T^2 \times [0,1]$, then output ``yes'', where $T^2$ is the torus. Otherwise output ``no''.
    \end{enumerate}
    The knot $K$ is not the core of $V$ if and only if $E$ is not homeomorphic to $T^2 \times [0,1]$.
    Thus, this algorithm outputs ``yes'' if and only if $K$ is not the core of $V$.
    By Lemma \ref{lem:rotationNum} and Lemma \ref{lem:exterior}, the $1$st step and the $2$nd step run in polynomial time of $n$.
    In the $3$rd step, we see that $|r(K)| = 1$ since if $|r(K)| \neq 1$, then the $1$st step outputs ``yes''.
    This implies that $E$ is irreducible by Lemma \ref{lem:irr}.
    Using Theorem \ref{thm:thickSurco_NP}, the $3$rd step is performed in non-deterministic polynomial time of $n$.
    Since there is a non-deterministic polynomial time algorithm that decides whether $K$ is not the core of $V$,
    the solid torus core recognition problem is in $\textbf{co-NP}$.
\end{proof}

\subsection{The Hopf link recognition problem}
In this subsection, we show that the Hopf link recognition problem is in \textbf{NP}.
\begin{dfn}[The Hopf link recognition problem]
    Let $D$ be a diagram of a link $L$ in $\mathbb{S}^3$.
    The Hopf link recognition problem is a problem that, given $D$, decides $L$ is the Hopf link.
\end{dfn}
Let $D$ be a diagram of a link $L$ in $\mathbb{S}^3$.
Suppose that $c$ is the number of crossings of $D$ and $k$ is the number of components of $L$.
The \textit{crossing measure} $n$ of $D$ is defined as
\[
    n = c + k -1.
\]
The computational complexity of a problem whose input is a link diagram is measured by the crossing measure of the input diagram.
See \cite{HLP} for details.

\begin{lem}[Hass-Lagarias-Pippenger \cite{HLP}]\label{lem:alg3sphere}
    Let $D$ be a diagram of a link $L$ in $\mathbb{S}^3$ and $n$ be the crossing measure of $D$.
    Then there is a $\mathcal{O}(n \log n)$ time algorithm that, given $D$, outputs a triangulation $\mathcal{T}_L$ of $\mathbb{S}^3$ such that the $1$-skeleton $\mathcal{T}_L^{(1)}$ contains $L$.
    Furthermore, $\text{size}(\mathcal{T}_L)$ is at most $\mathcal{O}(n)$.
\end{lem}

\begin{cor}
    The Hopf link recognition problem is in \textbf{NP} $\cap$ \textbf{co-NP}.
\end{cor}
\begin{proof}
    Let $D$ be a diagram of a link $L$ in $\mathbb{S}^3$.
    Suppose that the crossing measure of $D$ is $n$.

    First, we show that the Hopf link recognition problem is in \textbf{NP}.
    We consider the following non-deterministic algorithm.
    \begin{enumerate}
        \item If the number of components of $L$ is two, then let $L = K_1 \cup K_2$. Otherwise output ``no''.
        \item If $K_1$ is the unknot in $\mathbb{S}^3$, then construct a triangulation $\mathcal{T}_{E_1}$ of the solid torus $E_1 = \mathbb{S}^3 - \text{int}N(K_1)$ such that $\mathcal{T}_{E_1}^{(1)}$ contains $K_2$. Otherwise output ``no''. 
        \item If $K_2$ is the core of $E_1$, then output ``yes''. Otherwise output ``no''.
    \end{enumerate}

    The link $L = K_1 \cup K_2$ is the Hopf link if and only if $K_1$ is the unknot and $K_2$ is the core of $E_1 = \mathbb{S}^3 - \text{int}N(K_1)$.
    Therefore, the above algorithm outputs ``yes'' if and only if $L$ is the Hopf link.

    The $1$st step is performed in $\mathcal{O}(n)$ time.
    Let $D_1$ be the diagram of $K_1$ that is contained in $D$.
    By Theorem \ref{thm:unknot}, we can determine whether $D_1$ is a diagram of the unknot in non-deterministic polynomial time of $n$.
    By Lemma \ref{lem:alg3sphere}, a triangulation $\mathcal{T}_L$ of $\mathbb{S}^3$ such that the $1$-skeleton $\mathcal{T}_L^{(1)}$ contains $L$ is constructed in polynomial time of $n$, 
    and $\text{size}(\mathcal{T}_L)$ is at most $\mathcal{O}(n)$.
    A triangulation $\mathcal{T}_{E_1}$ of $E_1 = \mathbb{S}^3 - \text{int}N(K_1)$ is obtained by barycentrically subdividing $\mathcal{T}_L^{(1)}$ twice and removing the tetrahedra containing $K_1$.
    This implies that $\mathcal{T}_{E_1}$ is obtained in polynomial time of $n$, and $\text{size}(\mathcal{T}_{E_1})$ is at most $\mathcal{O}(n)$.
    Thus, the $2$nd step runs in polynomial time of $n$.
    Using Theorem \ref{thm:main}, the $3$rd step is performed in non-deterministic polynomial time of $n$.
    Therefore, the above algorithm runs in non-deterministic polynomial time of $n$.
    Since there is a non-deterministic polynomial time algorithm for the Hopf link recognition problem, this is in \textbf{NP}.

    Next, we show that the Hopf link recognition problem is in \textbf{co-NP}.
    We consider the following non-deterministic algorithm.
    \begin{enumerate}
        \item If the number of components of $L$ is two, then let $L = K_1 \cup K_2$. Otherwise output ``yes''.
        \item If $K_1$ is not the unknot in $\mathbb{S}^3$, then output ``yes''. Otherwise construct a triangulation $\mathcal{T}_{E_1}$ of the solid torus $E_1 = \mathbb{S}^3 - \text{int}N(K_1)$ such that $\mathcal{T}_{E_1}^{(1)}$ contains $K_2$.
        \item If $K_2$ is not the core of $E_1$, then output ``yes''. Otherwise output ``no''.
    \end{enumerate}

    If $L$ is not the Hopf link, then $L$ satisfies one of the following:
    \begin{itemize}
        \item the number of components of $L$ is not two,
        \item $K_1$ is not the unknot, or
        \item $K_1$ is the unknot and $K_2$ is not the core of $E_1$.
    \end{itemize}
    Thus, the above algorithm outputs ``yes'' if $L$ is not the Hopf link.
    Conversely, if $L$ is the Hopf link, then this algorithm outputs ``no'' in the $3$rd step.
    Thus, this algorithm outputs ``yes'' if and only if $L$ is not the Hopf link.

    We see that the $1$st step runs in polynomial time of $n$.
    The $2$nd step is performed in non-deterministic polynomial time of $n$ since the unknot recognition is in \textbf{co-NP} (\cite{Lack}).
    Using Theorem \ref{thm:main2}, we see that the $3$rd step runs in non-deterministic polynomial time of $n$.
    Since there is a non-deterministic polynomial time algorithm that decides whether $L$ is not the Hopf link,
    the Hopf link recognition problem is in $\textbf{co-NP}$.
\end{proof}

\bibliographystyle{plain}
\bibliography{bibliography}

\end{document}